\documentclass[11pt]{article}
\usepackage{amssymb,amsmath}
\usepackage{color}
\usepackage{marginnote}

\newtheorem{thm}{Theorem}[section]

\newtheorem{prop}[thm]{Proposition}
\newtheorem{lem}[thm]{Lemma}
\newtheorem{rem}[thm]{Remark}
\newtheorem{rems}[thm]{Remarks}
\newtheorem{cor}[thm]{Corollary}
\newenvironment{proof}
{\begin{trivlist}\item[]{\bf Proof.}}%
{\hspace*{\fill}\QED\end{trivlist}}

\newcommand{\N}{\mathbb{N}}

\newcommand{\R}{\mathbb{R}}

\newcommand{\cF}{\mathcal{F}}
\newcommand{\cG}{\mathcal{G}}
\newcommand{\cI}{\mathcal{I}}

\newcommand{\cM}{\mathcal{M}}
\newcommand{\cO}{\mathcal{O}}

\newcommand{\dd}{\,{\rm d}}
\newcommand{\D}{{\rm d}}

\renewcommand{\div}{\mathop{\mathrm{div}}\nolimits}
\newcommand{\curl}{\mathop{\mathrm{curl}}}

\renewcommand{\:}{\thinspace :}

\newcommand{\1}{\mathbf{1}}

\newcommand{\lin}{\mathrm{lin}}
\newcommand{\weakto}{\rightharpoonup}
\newcommand{\tv}{\mathrm{tv}}
\newcommand{\pp}{\mathrm{pp}}

\newcommand{\DS}{\displaystyle}
\newcommand{\BMO}{\mathrm{BMO}}
\newcommand{\VMO}{\mathrm{VMO}}
\newcommand{\QED}{\mbox{}\hfill$\Box$}

\begin{document}

\title{Remarks on the Cauchy problem for the
axisymmetric Navier-Stokes equations\footnote{
This paper is dedicated to Prof. Denis Serre on the occasion of 
his sixtieth birthday.}}

\author{
\null\\
{\bf Thierry Gallay}\\
Institut Fourier\\
UMR CNRS 5582\\
Universit\'e de Grenoble I, BP 74\\
38402 Saint-Martin-d'H\`eres, France\\
{\tt Thierry.Gallay@ujf-grenoble.fr}
\and
\\
{\bf Vladim\'ir \v{S}ver\'ak}\\
School of Mathematics\\
University of Minnesota\\
127 Vincent Hall,
206 Church St.\thinspace SE\\
Minneapolis, MN 55455, USA\\
{\tt sverak@math.umn.edu}}


\maketitle

\begin{abstract}
Motivated by applications to vortex rings, we study the Cauchy problem
for the three-dimensional axisymmetric Navier-Stokes equations without
swirl, using scale invariant function spaces. If the axisymmetric
vorticity $\omega_\theta$ is integrable with respect to the
two-dimensional measure $\D r\dd z$, where $(r,\theta,z)$ denote the
cylindrical coordinates in $\R^3$, we show the existence of a unique
global solution, which converges to zero in $L^1$ norm as $t \to
\infty$.  The proof of local well-posedness follows exactly the same
lines as in the two-dimensional case, and our approach emphasizes the
similarity between both situations. The solutions we construct have
infinite energy in general, so that energy dissipation cannot be
invoked to control the long-time behavior. We also treat the more
general case where the initial vorticity is a finite measure whose
atomic part is small enough compared to viscosity. Such data include
point masses, which correspond to vortex filaments in the
three-dimensional picture.
\end{abstract}

\section{Introduction}\label{sec1}

Among all three-dimensional incompressible flows, {\em axisymmetric
flows without swirl} form a particular class that is relatively
simple to study and yet contains interesting examples, such as
circular vortex filaments or toroidal vortex rings.  For the
evolutions defined by both the Euler and the Navier-Stokes equations,
global well-posedness in that class was established almost fifty years
ago by Ladyzhenskaya \cite{La} and Ukhovksii \& Yudovich \cite
{UY}. In the viscous case, the original approach of \cite{La,UY}
applies to velocity fields in the Sobolev space $H^2(\R^3)$, see
\cite{LMNP}, but it is possible to obtain the same conclusions under
the weaker assumption that the initial velocity belongs to
$H^{1/2}(\R^3)$ \cite{Ab}. In all these works, global existence
for arbitrary large data is shown by combining the standard energy
estimate, which holds for general solutions of the Navier-Stokes
equations, with a priori bounds on the vorticity that are specific
to the axisymmetric case.

In this paper, we revisit the Cauchy problem for the axisymmetric
Navier-Stokes equations (without swirl) for the following
reasons. First, motivated by a future study of vortex filaments, we
wish to formulate a global well-posedness result involving {\em scale
invariant} function spaces only. As was already mentioned, all
previous works deal with finite energy solutions, and energy is not a
scale invariant quantity for the three-dimensional viscous flows. In
particular, the long-time behavior of axisymmetric solutions has not
been studied in terms of scale invariant norms. Our second motivation
is to emphasize the analogy between the axisymmetric case and the
two-dimensional situation where the velocity field is planar and
depends on two variables only. Indeed, we shall see that, if
appropriate function spaces are used, local existence of solutions can
be established in the axisymmetric case using literally the same proof
as in the two-dimensional situation, which has been studied by many
authors \cite{BA,GMO,Ka2}. However, significant differences appear
when one considers a priori estimates and long-time asymptotics.

To formulate our results, we introduce some notation. The time
evolution of viscous incompressible flows is described by the
Navier-Stokes equations
\begin{equation}\label{NS3D}
  \partial_t u + (u\cdot \nabla)u \,=\, \Delta u - \nabla p~,
  \qquad \div u \,=\, 0~,
\end{equation}
where $u = u(x,t) \in \R^3$ denotes the velocity field and $p = p(x,t) \in \R$
the pressure field. For simplicity, we assume throughout this paper
that the kinematic viscosity and the fluid density are both equal to $1$.
We restrict ourselves to axisymmetric solutions without swirl for which
the velocity field has the following particular form\:
\begin{equation}\label{uaxi}
  u(x,t) \,=\, u_r(r,z,t) e_r + u_z(r,z,t) e_z~.
\end{equation}
Here $(r,\theta,z)$ are the usual cylindrical coordinates in $\R^3$,
defined by $x = (r\cos\theta,r\sin\theta,z)$ for any $x \in \R^3$,
and $e_r, e_\theta, e_z$ denote the unit vectors in the radial, toroidal,
and vertical directions, respectively\:
\[
  e_r \,=\, \begin{pmatrix}\cos\theta \\ \sin\theta \\0\end{pmatrix}~,
  \qquad e_\theta \,=\, \begin{pmatrix}-\sin\theta \\ \cos\theta \\0
  \end{pmatrix}~, \qquad e_z \,=\, \begin{pmatrix}0 \\ 0 \\1\end{pmatrix}~.
\]
We emphasize that the ``swirl''  $u \cdot e_\theta$ is assumed to vanish
identically. This means that the velocity field \eqref{uaxi} is not only
invariant under rotations about the vertical axis, but also under 
reflections by any plane containing the vertical axis. 

A direct calculation shows that the vorticity $\omega = \curl u$ associated
with the velocity field \eqref{uaxi} is purely toroidal\:
\begin{equation}\label{omaxi}
  \omega(r,z,t) \,=\, \omega_\theta(r,z,t)e_\theta~, \qquad
  \hbox{where}\quad \omega_\theta = \partial_z u_r - \partial_r u_z~.
\end{equation}
The flow is entirely determined by the single quantity
$\omega_\theta$, because the velocity field $u$ can be reconstructed
by solving the linear elliptic system
\begin{equation}\label{diffBS}
  \partial_r u_r + \frac{1}{r} u_r + \partial_z u_z \,=\, 0~,
  \qquad \partial_z u_r - \partial_r u_z \,=\, \omega_\theta~,
\end{equation}
in the half space $\Omega = \{(r,z) \in \R^2\,|\, r > 0\,,~z \in \R\}$,
with boundary conditions $u_r = \partial_r u_z = 0$ at $r = 0$.
System \eqref{diffBS} is the differential formulation of the axisymmetric
Biot-Savart law, which will be studied in more detail in Section~\ref{sec2}
below.

The evolution equation for $\omega_\theta$ reads
\begin{equation}\label{omeq}
  \partial_t \omega_\theta + u\cdot\nabla\omega_\theta -
  \frac{u_r}{r}\omega_\theta \,=\, \Delta \omega_\theta
  - \frac{\omega_\theta}{r^2}~,
\end{equation}
where $u\cdot\nabla = u_r\partial_r + u_z\partial_z$ and $\Delta =
\partial_r^2 + \frac{1}{r}\partial_r + \partial_z^2$ denotes the
Laplace operator in cylindrical coordinates. Equation \eqref{omeq} is
considered in the half-plane $\Omega$ with homogeneous Dirichlet condition
at the boundary $r = 0$. As was already observed in \cite{La,UY}, it is
useful to consider also the related quantity
\begin{equation}\label{etadef}
  \eta(r,z,t) \,=\, \frac{\omega_\theta(r,z,t)}{r}~,
\end{equation}
which satisfies the advection-diffusion equation
\begin{equation}\label{etaeq}
  \partial_t \eta + u\cdot\nabla\eta \,=\, \Delta \eta
  + \frac{2}{r}\partial_r\eta~,
\end{equation}
with homogeneous Neumann condition at the boundary $r = 0$. Systems
\eqref{omeq} and \eqref{etaeq} are, of course, perfectly
equivalent. In what follows we find it more convenient to work with
the axisymmetric vorticity equation \eqref{omeq}, at least to prove
local existence of solutions, but equation \eqref{etaeq} will be
useful to derive a priori estimates and to study the long-time
behavior.

Throughout this paper, to emphasize the similarity with the
two-dimensional case, we equip the half-plane $\Omega = \{(r,z)
\in \R^2\,|\, r > 0\,,~z \in \R\}$ with the {\em two-dimensional}
measure $\D r\dd z$, as opposed to the 3D measure $r\dd r\dd z$
which could appear more natural for axisymmetric problems.
Thus, given any $p \in [1,\infty)$, we denote by $L^p(\Omega)$ the
space of measurable functions $\omega_\theta : \Omega \to \R$ for
which the following norm is finite\:
\[
  \|\omega_\theta\|_{L^p(\Omega)} \,=\, \left(\int_\Omega
  |\omega_\theta(r,z)|^p \dd r\dd z\right)^{1/p}~, \qquad
  1 \le p < \infty~.
\]
The space $L^\infty(\Omega)$ is defined similarly. Sometimes,
however, it is more convenient to use the 3D measure $r\dd r\dd z$,
and the corresponding spaces are then denoted by $L^p(\R^3)$ to
avoid confusion. For instance, we define
\[
  \|\eta\|_{L^p(\R^3)} \,=\, \left(\int_\Omega
  |\eta(r,z)|^p \,r \dd r\dd z\right)^{1/p}~, \qquad
  1 \le p < \infty~.
\]

We are now in position to state our first main result.

\begin{thm}\label{main1}
For any initial data $\omega_0 \in L^1(\Omega)$, the
axisymmetric vorticity equation \eqref{omeq} has a
unique global mild solution
\begin{equation}\label{omprop}
  \omega_\theta \in C^0([0,\infty),L^1(\Omega)) \cap
  C^0((0,\infty),L^\infty(\Omega))~.
\end{equation}
The solution satisfies $\|\omega_\theta(t)\|_{L^1(\Omega)} \le 
\|\omega_0\|_{L^1(\Omega)}$ for all $t > 0$, and
\begin{align}\label{asym1}
  \lim_{t\to 0} t^{1-\frac1p}\|\omega_\theta(t)\|_{L^p(\Omega)}
  \,&=\, 0~, \qquad \hbox{for}\quad 1 < p \le \infty~, \\ \label{asym2}
  \lim_{t\to \infty} t^{1-\frac1p}\|\omega_\theta(t)\|_{L^p(\Omega)}
  \,&=\, 0~, \qquad \hbox{for}\quad 1 \le p \le \infty~.
\end{align}
If, in addition, the axisymmetric vorticity is non-negative 
and has finite impulse\:
\begin{equation}\label{cM-def}
 \cI \,=\, \int_\Omega r^2 \omega_0(r,z)\dd r\dd z \,<\, \infty~,
\end{equation} 
then
\begin{equation}\label{Mnon-V}
  \lim_{t \to \infty} t^2 \omega_\theta(r\sqrt{t},z\sqrt{t},t) \,=\,
  \frac{\cI}{16\sqrt{\pi}}\,r\,e^{-\frac{r^2+z^2}{4}}~, \qquad
  (r,z) \in \Omega~,
\end{equation}
where convergence holds in $L^p(\Omega)$ for $1 \le p \le \infty$.
In particular $\|\omega_\theta(t)\|_{L^p(\Omega)} = \cO(t^{-2+\frac1p})$
as $t \to \infty$ in that case.
\end{thm}

\begin{rem}\label{mildrem}
A mild solution of \eqref{omeq} on $\R_+ = [0,\infty)$ is a solution
of the associated integral equation, namely Eq.~\eqref{omint} below.
As will be verified in Section~\ref{sec4}, both sides of \eqref{omint}
are well-defined if $\omega_\theta$ satisfies \eqref{omprop} and
\eqref{asym1}. In the uniqueness claim, we only assume that
$\omega_\theta$ satisfies \eqref{omprop} and is a mild solution
of \eqref{omeq} for strictly positive times. The proof then
shows that \eqref{asym1} automatically holds.
\end{rem}

The statement of Theorem~\ref{main1} has several aspects, and it is
worth discussing them separately. The {\it local well-posedness} claim
is certainly not surprising, because the class of initial data we
consider is covered by at least two existence results in the
literature. Indeed, if $\omega_\theta \in L^1(\Omega)$, it is easy
to verify that the vorticity $\omega = \omega_\theta e_\theta$
belongs to the Morrey space $M^{3/2}(\R^3)$ defined by the norm
\[
  \|\omega\|_{M^{3/2}} \,=\, \sup_{x \in \R^3}\,
  \sup_{R >0}\,\frac{1}{R}\int_{B(x,R)} |\omega(x)|\dd x~,
\]
where $B(x,R) \subset \R^3$ denotes the ball or radius $R > 0$
centered at $x \in \R^3$. In addition $\omega$ can be approximated in
$M^{3/2}(\R^3)$ by smooth and compactly supported functions. As was
proved by Giga \& Miyakawa \cite{GM}, the Navier-Stokes equations in
$\R^3$ thus have a unique local solution with initial vorticity
$\omega$, which is even global in time if the norm $\|\omega\|_{M^{3/2}}$
is sufficiently small. On the other hand, under the same assumption on
the vorticity, one can show that the velocity field $u$ given by the
Biot-Savart law in $\R^3$ belongs to the space $\BMO^{-1}(\R^3)$
defined by the norm
\[
  \|u\|_{\BMO^{-1}} \,=\, \sup_{x \in \R^3}\,\sup_{R >0} \left(\frac{1}{R^3}
  \int_{B(x,R)} \int_0^{R^2} |e^{t\Delta}u|^2\dd t\dd x\right)^{1/2}~,
\]
where $e^{t\Delta}$ denotes the heat semigroup in $\R^3$. In fact $u$
can be approximated by smooth and compactly supported functions in
$\BMO^{-1}(\R^3)$, so that $u \in \VMO^{-1}(\R^3)$.  Thus we can also
invoke the celebrated result by Koch \& Tataru \cite{KT} to obtain the
existence of a unique local solution to the Navier-Stokes equation in
$\R^3$, which is again global in time if the norm $\|u\|_{\BMO^{-1}}$
is sufficiently small. In contrast to the general results in
\cite{GM,KT}, the approach we follow to solve the Cauchy problem
for Eq.~\eqref{omeq} uses specific features of the axisymmetric case.
As we shall see in Section~\ref{sec4}, it is elementary and
completely parallel to the two-dimensional situation which was
studied e.g. in \cite{BA,Ga}.

The assertion of {\em global well-posedness} in Theorem~\ref{main1} is
also quite natural in view of the historical results by Ladyzhenskaya
\cite{La} and Ukhovkii \& Yudovich \cite{UY}. As in \cite{La,UY} we
use the structure of equation \eqref{etaeq} to derive a priori
estimates on the quantity $\eta$ in $L^p(\R^3)$, for $1 \le p \le \infty$.
However, since the solutions we consider do not have finite energy in
general, we cannot apply the classical energy estimate to obtain a
uniform bound on the velocity field in $L^2(\R^3)$. Instead we prove
that any solution of \eqref{omeq} with initial data $\omega_0 \in
L^1(\Omega)$ satisfies $\|\omega_\theta(t)\|_{L^1(\Omega)} \le
\|\omega_0\|_{L^1(\Omega)}$ for all $t > 0$ and
\begin{equation}\label{newom}
  \sup_{t > 0}\,t \|\omega_\theta(t)\|_{L^\infty(\Omega)} \,\le\,
  C\bigl(\|\omega_0\|_{L^1(\Omega)}\bigr)~,
\end{equation}
where $C(s) = \cO(s)$ as $s \to 0$. This new a priori estimate
is scale invariant, and implies that all solutions of \eqref{omeq}
with initial data in $L^1(\Omega)$ are global. Using the axisymmetric
Biot-Savart law, we also deduce the following optimal bound on
the velocity field\:
\begin{equation}\label{newu}
  \sup_{t > 0}\,t^{1/2} \|u(t)\|_{L^\infty(\Omega)} \,\le\,
  C\bigl(\|\omega_0\|_{L^1(\Omega)}\bigr)~.
\end{equation}

Our last comment on Theorem~\ref{main1} concerns the {\em long-time
  behavior}, which differs significantly from what happens in the
two-dimensional case. In the latter situation, the $L^1$ norm of the
vorticity is non-increasing in time, but does not converge to zero in
general (in particular, it is constant for solutions with a definite
sign). The long-time behavior is described by self-similar solutions,
called Oseen vortices, which have a nonzero total circulation
\cite{GW1,GW2}. In contrast, the axisymmetric vorticity
$\omega_\theta$ vanishes on the boundary of the half-plane $\Omega$,
so that the $L^1$ norm is strictly decreasing for all nontrivial
solutions. As we shall see in Section~\ref{sec6}, this implies that
$\|\omega_\theta(t)\|_{L^1(\Omega)} \to 0$ when $t \to \infty$, as
asserted in \eqref{asym2}. In other words, the long-time behavior of
the axisymmetric vorticity is trivial when measured in scale invariant
function spaces. More can be said when the initial data have a
definite sign and a finite impulse, given by \eqref{cM-def}. In that
case, the axisymmetric vorticity $\omega_\theta$ inherits the same
properties for all positive times and converges as $t \to \infty$ to a
self-similar solution of the linearized equation \eqref{omlin}, whose
profile is explicitly determined in \eqref{Mnon-V}.

As in the two-dimensional case, it is possible to extend the local
existence claim in Theorem~\ref{main1} to a larger class of
initial data, so as to include finite measures as initial vorticities.
Let $\cM(\Omega)$ denote the set of all real-valued finite measures
on the half-plane $\Omega$, equipped with the total variation norm
\[
  \|\mu\|_\tv \,=\, \sup\left\{ \int_\Omega \phi \dd\mu \,\Big|\,
   \phi \in C_0(\Omega)\,,~ \|\phi\|_{L^\infty(\Omega)} \le 1
   \right\}~, \qquad \hbox{for}~ \mu \in \cM(\Omega)~,
\]
where $C_0(\Omega)$ is the set of all real-valued continuous functions
on $\Omega$ that vanish at infinity and on the boundary
$\partial\Omega$. If $\mu \in \cM(\Omega)$ is absolutely continuous
with respect to Lebesgue's measure, then $\mu = \omega_\theta \dd r\dd
z$ for some $\omega_\theta \in L^1(\Omega)$, and $\|\mu\|_\tv =
\|\omega_\theta\|_{L^1(\Omega)}$. More generally, one can decompose
any $\mu \in \cM(\Omega)$ as $\mu = \mu_{ac} + \mu_{sc} + \mu_{pp}$,
where $\mu_{ac}$ is absolutely continuous with respect to Lebesgue's
measure, $\mu_{pp}$ is a countable collection of Dirac masses, and
$\mu_{sc}$ has no atoms but is supported on a set of zero Lebesgue
measure.  In the original three-dimensional picture, each Dirac mass
in the atomic part $\mu_{pp}$ corresponds to a circular vortex
filament, whereas vortex sheets are included in the singularly
continuous part $\mu_{sc}$.

The proof of Theorem~\ref{main1} can be adapted to initial vorticities
in $\cM(\Omega)$, and gives the following statement, which is our
second main result.

\begin{thm}\label{main2}
There exist positive constants $\epsilon$ and $C$ such that,
for any initial data $\omega_0 \in \cM(\Omega)$ with
$\|(\omega_0)_\pp\|_\tv \le \epsilon$, the axisymmetric vorticity
equation \eqref{omeq} has a unique global mild solution
$\omega_\theta \in C^0((0,\infty),L^1(\Omega) \cap
L^\infty(\Omega))$ satisfying
\begin{equation}\label{main2prop}
   \limsup_{t\to 0}\|\omega_\theta(t)\|_{L^1(\Omega)} \,<\, \infty~, \qquad
   \limsup_{t\to 0} t^{1/4}\|\omega_\theta(t)\|_{L^{4/3}(\Omega)} 
   \,\le\, C\epsilon~,
\end{equation}
and such that $\omega_\theta(t) \weakto \omega_0$ as $t \to 0$.
Moreover, the asymptotic estimates for $t\to\infty$ given in 
Theorem~\ref{main1} hold without change. 
\end{thm}

Observe that we now have a limitation on the size of the data, which
however only affects the atomic part of the initial vorticity.  This
technical restriction inevitably occurs if local existence is
established using a fixed point argument in scale invariant spaces, as
we do in Section~\ref{sec4}. In the two-dimensional case, early
results by Giga, Miyakawa, \& Osada \cite{GMO} and by Kato \cite{Ka2}
had a similar limitation, which was then relaxed in \cite{GW2,GG}
using completely different techniques.  In the axisymmetric situation,
existence of a global solution to \eqref{omeq} with a large Dirac mass
as initial vorticity has recently been established by Feng and \v
Sver\'ak \cite{FS}, using an approximation argument, but uniqueness is
still under investigation. For a general initial vorticity $\omega_0
\in \cM(\Omega)$, both existence and uniqueness are open.

Even if we restrict ourselves to initial vorticities with a small
atomic part, the uniqueness claim in Theorem~\ref{main2} is probably
not optimal. Indeed, although the solutions we construct satisfy both
estimates in \eqref{main2prop}, we believe that uniqueness should hold
(as in the two-dimensional case) under the sole assumptions that
$\omega_\theta(t)$ is uniformly bounded in $L^1(\Omega)$ for $t > 0$
and converges weakly to the initial vorticity $\omega_0$ as $t \to 0$.
However, technical difficulties arise when adapting the
two-dimensional proof to the axisymmetric case, and for the moment we
need an additional assumption, such as the second estimate in
\eqref{main2prop}, to obtain uniqueness. We hope to clarify that
question in a future work.

The rest of this paper is organized as follows. In Section~\ref{sec2},
which is devoted to the axisymmetric Biot-Savart law, we estimate
various norms of the velocity field $u$ in terms of the axisymmetric
vorticity $\omega_\theta$. In Section~\ref{sec3}, we show that the
semigroup generated by the linearization of \eqref{omeq} about the
origin satisfies the same $L^p - L^q$ estimates as the two-dimensional
heat kernel. After these preliminaries, we prove the local existence
claims in Theorems~\ref{main1} and \ref{main2} in Sections~\ref{sec41}
and \ref{sec42}, respectively. Global existence follows from a priori
estimates which are established in Section~\ref{sec5}. Finally, the
long-time behavior is investigated in Section~\ref{sec6}. We prove
that all solutions of \eqref{omeq} converge to zero in $L^1(\Omega)$,
and we also compute the leading term in the long-time asymptotics for
vorticities with a definite sign and a finite impulse.

\medskip\noindent{\bf Acknowledgements.} This project started during
visits of the first named author to the University of Minnesota, 
whose hospitality is gratefully acknowledged. Our research was 
supported in part by grants DMS 1362467 and DMS 1159376 from the 
National Science Foundation (VS) and by the grant ``Dyficolti'' 
ANR-13-BS01-0003-01 from the French Ministry of Research (ThG). 

\section{The axisymmetric Biot-Savart law}\label{sec2}

In this section, we assume that the axisymmetric vorticity
$\omega_\theta : \Omega \to \R$ is given, and we study the properties
of the velocity field $u = (u_r,u_z)$ satisfying the linear elliptic
system \eqref{diffBS}. The divergence-free condition $\partial_r(r
u_r) + \partial_z (ru_z) = 0$ implies that there exists a function
$\psi : \Omega \to \R$ such that
\begin{equation}\label{upsi}
  u_r \,=\, -\frac{1}{r}\frac{\partial \psi}{\partial z}~,
  \qquad
  u_z \,=\, \frac{1}{r}\frac{\partial \psi}{\partial r}~.
\end{equation}
As $\partial_z u_r - \partial_r u_z =\omega_\theta$, the {\em
axisymmetric stream function} $\psi$ satisfies the following linear
elliptic equation in the half-space $\Omega$\:
\begin{equation}\label{psieq}
  -\partial_r^2 \psi + \frac{1}{r}\partial_r \psi - \partial_z^2 \psi
  \,=\, r \omega_\theta~.
\end{equation}
Boundary conditions for \eqref{psieq} are determined by observing
that, for a smooth axisymmetric vector field $u : \R^3 \to \R^3$ with
$\div u = 0$, the stream function defined by \eqref{upsi} satisfies
the asymptotic expansion
\begin{equation}\label{psireg}
  \psi(r,z) \,=\, \psi_0 + r^2 \psi_2(z) + \cO(r^4)~, \qquad
  \hbox{as } r \to 0~,
\end{equation}
see \cite{LW} for an extensive discussion of these regularity
issues. Without loss of generality, we can assume that the constant
$\psi_0$ in \eqref{psireg} is equal to zero, in which case we
conclude that $\psi(0,z) = \partial_r \psi(0,z) = 0$.

The solution of \eqref{psieq} with these boundary conditions
is well known, see e.g. \cite{FS}. If we assume that the vorticity
$\omega_\theta$ decays sufficiently fast at infinity, we have
the explicit representation
\begin{equation}\label{BSpsi}
  \psi(r,z) \,=\, \frac{1}{2\pi}\int_\Omega \sqrt{r \bar r}\,
  F\left(\frac{(r-\bar r)^2 + (z-\bar z)^2}{r\bar r}\right)
  \omega_\theta(\bar r,\bar z)\dd\bar r \dd\bar z~,
\end{equation}
where the function $F : (0,\infty) \to \R$ is defined by
\begin{equation}\label{Fdef}
  F(s) \,=\, \int_0^\pi \frac{\cos\phi\,\dd\phi}{\bigl({2(1{-}\cos\phi)+s}
  \bigr)^{1/2}} \,=\, \int_0^{\pi/2} \frac{\cos(2\phi)\,\dd\phi}{
  \bigl(\sin^2\phi+s/4\bigr)^{1/2}}~, \qquad s > 0~.
\end{equation}
Useful properties of $F$ are collected in the following
lemma, whose proof can be found in \cite[Section~19]{Sv}.

\begin{lem}\label{Fproperties}
The function $F : (0,\infty) \to \R$ defined by \eqref{Fdef}
is decreasing and satisfies the asymptotic expansions\:
\\[2mm]
i) $\DS F(s) = \log\Bigl(\frac{8}{\sqrt{s}}\Bigr) -2 + \cO\Bigl(s
\log\frac1s\Bigr)$ and $\DS F'(s) = -\frac{1}{2s} + \cO\Bigl(
\log\frac1s\Bigr)$ as $s \to 0$; \\[2mm]
ii) $\DS F(s) = \frac{\pi}{2s^{3/2}} + \cO\Bigl(\frac{1}{s^{5/2}}\Bigr)$
and $\DS F'(s) = -\frac{3\pi}{4s^{5/2}} + \cO\Bigl(\frac{1}{s^{7/2}}\Bigr)$
as $s \to \infty$.
\end{lem}

\begin{rem}\label{Fbounds}
It follows in particular from Lemma~\ref{Fproperties} that
the maps $s \mapsto s^\alpha F(s)$ and $s \mapsto s^\beta F'(s)$ are
bounded if $0 < \alpha \le 3/2$ and $1 \le \beta \le 5/2$. These
observations will be constantly used in the subsequent proofs.
\end{rem}

Combining \eqref{upsi} and \eqref{BSpsi}, we obtain explicit formulas
for the axisymmetric Biot-Savart law\:
\begin{equation}\label{BSu}
  u_r(r,z) \,=\, \int_\Omega G_r(r,z,\bar r,\bar z)\omega_\theta(\bar r,
  \bar z)\dd\bar r
  \dd\bar z~, \qquad
  u_z(r,z) \,=\, \int_\Omega G_z(r,z,\bar r,\bar z)\omega_\theta(\bar r,
  \bar z)\dd\bar r\dd\bar z~,
\end{equation}
where
\begin{align}\label{Grdef}
  G_r(r,z,\bar r,\bar z) \,&=\, -\frac{1}{\pi}\frac{z-\bar z}{r^{3/2}
  \,\bar r^{1/2}}\,F'(\xi^2)~, \qquad \xi^2 \,=\, \frac{(r-\bar r)^2 +
  (z-\bar z)^2}{r\bar r}~, \\[1mm] \label{Gzdef}
  G_z(r,z,\bar r,\bar z) \,&=\, \frac{1}{\pi}\frac{r-\bar r}{r^{3/2}
  \,\bar r^{1/2}} \,F'(\xi^2) + \frac{1}{4\pi}\frac{\bar r^{1/2}}{r^{3/2}}
  \Bigl(F(\xi^2) - 2\xi^2 F'(\xi^2)\Bigr)~.
\end{align}

Our first result gives elementary estimates for the axisymmetric
Biot-Savart law in usual Lebesgue spaces. We emphasize the striking
similarity with the corresponding bounds for the two-dimensional
Biot-Savart law in the plane $\R^2$, see e.g.  \cite[Lemma~2.1]{GW1}.

\begin{prop}\label{BSprop}
The following properties hold for the velocity field $u$ defined
from the vorticity $\omega_\theta$ via the axisymmetric Biot-Savart
law \eqref{BSu}. \\[1mm]
i) Assume that $1 < p < 2 < q < \infty$ and $\frac1q = \frac1p
- \frac12$. If $\omega_\theta \in L^p(\Omega)$, then $u \in
L^q(\Omega)^2$ and
\begin{equation}\label{BSest1}
  \|u\|_{L^q(\Omega)} \,\le\, C \|\omega_\theta\|_{L^p(\Omega)}~.
\end{equation}
ii) If $1 \le p < 2 < q \le \infty$ and $\omega_\theta \in L^p(\Omega)
\cap L^q(\Omega)$, then $u \in L^\infty(\Omega)^2$ and
\begin{equation}\label{BSest2}
  \|u\|_{L^\infty(\Omega)} \,\le\, C \|\omega_\theta\|_{L^p(\Omega)}^\sigma
  \,\|\omega_\theta\|_{L^q(\Omega)}^{1-\sigma}~,  \qquad \hbox{where}\quad
  \sigma \,=\, \frac{p}{2}\,\frac{q-2}{q-p} \,\in\, (0,1)~.
\end{equation}
\end{prop}

\begin{proof}
Both assertions follow from the basic estimate
\begin{equation}\label{Gest}
  |G_r(r,z,\bar r,\bar z)| + |G_z(r,z,\bar r,\bar z)|
  \,\le\, \frac{C}{\bigl((r-\bar r)^2 + (z-\bar z)^2\bigr)^{1/2}}~,
\end{equation}
which holds for all $(r,z) \in \Omega$ and all $(\bar r, \bar z) \in
\Omega$. At a heuristic level, estimate \eqref{Gest} follows quite
naturally from the scaling properties of $G_r(r,z,\bar r, \bar z)$ and
$G_z(r,z,\bar r,\bar z)$ if we observe that these functions behave for
$(r,z)$ close to $(\bar r,\bar z)$ like the two-dimensional velocity
field generated by a Dirac mass located at $(\bar r, \bar z)$ in
$\R^2$. To prove \eqref{Gest} rigorously, we first bound the radial
component $G_r$. We distinguish two cases\:

\noindent a) If $\bar r \le 2r$, we use the fact that $\xi^2 F'(\xi^2)$
is bounded, and obtain the estimate
\begin{align*}
  |G_r(r,z,\bar r,\bar z)| \,&\le\, C\,\frac{|z-\bar z|}{r^{3/2}
  \,\bar r^{1/2}}\,\frac{r\bar r}{(r-\bar r)^2 + (z-\bar z)^2} \\
  \,&=\, C\,\frac{\bar r^{1/2}}{r^{1/2}}\,\frac{|z-\bar z|}{(r-\bar r)^2
  + (z-\bar z)^2} \,\le\, \frac{C}{\bigl((r-\bar r)^2 + (z-\bar z)^2
  \bigr)^{1/2}}~,
\end{align*}
because $\bar r/r \le 2$ and $|z-\bar z| \le \bigl((r-\bar r)^2 +
(z-\bar z)^2\bigr)^{1/2}$.

\noindent b) If $\bar r > 2r$, observing that $\xi^3 F'(\xi^2)$
is bounded, we deduce
\begin{align*}
  |G_r(r,z,\bar r,\bar z)| \,&\le\, C\,\frac{|z-\bar z|}{r^{3/2}
  \,\bar r^{1/2}}\,\frac{r^{3/2}\,\bar r^{3/2}}{\bigl((r-\bar r)^2 +
  (z-\bar z)^2\bigr)^{3/2}} \\
  \,&=\, C\,\frac{\bar r\,|z-\bar z|}{\bigl((r-\bar r)^2
  + (z-\bar z)^2\bigr)^{3/2}} \,\le\, \frac{C}{\bigl((r-\bar r)^2
  + (z-\bar z)^2\bigr)^{1/2}}~,
\end{align*}
because $\bar r < 2 (\bar r -r ) \le 2 \bigl((r-\bar r)^2 + (z-
\bar z)^2\bigr)^{1/2}$. This proves estimate \eqref{Gest} for
$G_r$.

\smallskip Similar calculations give the same bound for the second
component $G_z$ too. Indeed, the first term in the right-hand side of
\eqref{Gzdef} can be estimated exactly as above, using the obvious
fact that $|r-\bar r| \le \bigl((r-\bar r)^2 + (z-\bar z)^2\bigr)^{1/2}$.
The second term involves the quantity $F(\xi^2) -2\xi^2 F'(\xi^2)$,
which is bounded by $C\xi^{-1}$ in case a) and by $C\xi^{-3}$ in
case b). This concludes the proof of \eqref{Gest}.

Now, since the integral kernel $G = (G_r,G_z)$ in \eqref{BSu}
satisfies the same bound as the two-dimensional Biot-Savart kernel in
$\R^2$, properties i) and ii) in Proposition~\ref{BSprop} can be
established exactly as in the 2D case. Estimate \eqref{BSest1} thus
follows from the Hardy-Littlewood-Sobolev inequality, and the bound
\eqref{BSest2} can be proved by splitting the integration domain
and applying H\"older's inequality, see e.g. \cite[Lemma~2.1]{GW1}.
\end{proof}

The proof of Proposition~\ref{BSprop} shows that the axisymmetric
Biot-Savart law \eqref{BSu} has the same properties as the
usual Biot-Savart law in the whole plane $\R^2$. In fact, it
is possible to obtain in the axisymmetric situation weighted
inequalities (involving powers of the distance $r$ to the
vertical axis) which have no analogue in the 2D case. As an
example, we state here an interesting extension of estimate
\eqref{BSest1}.

\begin{prop}\label{BSprop2}
Let $\alpha,\beta \in [0,2]$ be such that $0 \le \beta-\alpha < 1$,
and assume that $p,q \in (1,\infty)$ satisfy
\[
  \frac{1}{q} \,=\, \frac{1}{p} - \frac{1+\alpha-\beta}{2}~.
\]
If $r^\beta \omega_\theta \in L^p(\Omega)$, then $r^\alpha u \in
L^q(\Omega)^2$ and we have the bound
\begin{equation}\label{BSest3}
  \|r^\alpha u\|_{L^q(\Omega)} \,\le\, C \|r^\beta \omega_\theta\|_{L^p(\Omega)}~.
\end{equation}
\end{prop}

\begin{proof}
As in the proof of Proposition~\ref{BSprop}, all we need is to
establish the pointwise estimate
\begin{equation}\label{Gest2}
 \frac{r^\alpha}{\bar r^\beta} \Bigl(|G_r(r,z,\bar r,\bar z)| +
 |G_z(r,z,\bar r,\bar z)|\Bigr) \,\le\, \frac{C}{\bigl(
 (r-\bar r)^2 + (z-\bar z)^2\bigr)^\lambda}~,
\end{equation}
for some $\lambda \in (0,1)$. Indeed, once \eqref{Gest2} is known, 
the bound \eqref{BSest3} follows immediately from the
Hardy-Littlewood-Sobolev inequality if $q^{-1} = p^{-1} + \lambda
-1$. To prove \eqref{Gest2}, we proceed as before and distinguish
three regions\: a) $\bar r/2 \le r \le 2 \bar r\,$; b) $r >
2\bar r\,$; c) $\bar r > 2r\,$. In each region, we bound the
quantities $F'(\xi^2)$ and $F(\xi^2) -2\xi^2 F'(\xi^2)$ appropriately
using Remark~\ref{Fbounds}. These calculations shows that inequality
\eqref{Gest2} holds with $\lambda = (1+\beta-\alpha)/2$ if and only if
the exponents $\alpha,\beta$ satisfy $0 \le \alpha \le \beta \le
2$. Since we need $\lambda < 1$ we assume in addition that
$\beta-\alpha < 1$. The details are straightforward and can be left to
the reader.
\end{proof}

\begin{rems}~\\
{\bf 1.} Similarly, one can establish a weighted analogue
of inequality \eqref{BSest2}. \\[1mm]
{\bf 2.} If $u$ is replaced by $u_r$, inequality \eqref{BSest3}
holds for all $\alpha,\beta \in [-1,2]$ such that $0 \le
\beta-\alpha < 1$. \\[1mm]
{\bf 3.} If we take $\alpha = 1/q$ and $\beta = 1/p$, so that
$\frac1q = \frac1p - \frac13$, inequality \eqref{BSest3} is
equivalent to
\[
  \|u\|_{L^q(\R^3)} \,\le\, C \|\omega\|_{L^p(\R^3)}~,
\]
which is well known for the three-dimensional Biot-Savart law,
see e.g. \cite[Lemma~2.1]{GW3}.
\end{rems}

Finally, we prove other weighted inequalities, which are
similar to those considered in \cite{FS}.

\begin{prop}\label{BSprop3}
The following estimates hold\:
\begin{align}\label{BSest4}
  \|u\|_{L^\infty(\Omega)} \,&\le\, C \|r \omega_\theta\|_{L^1(\Omega)}^{1/2}
  \,\|\omega_\theta/r\|_{L^\infty(\Omega)}^{1/2}~, \\ \label{BSest5}
  \Bigl\|\frac{u_r}{r}\Bigr\|_{L^\infty(\Omega)} \,&\le\, C
  \|\omega_\theta\|_{L^1(\Omega)}^{1/3}
  \,\|\omega_\theta/r\|_{L^\infty(\Omega)}^{2/3}~.
\end{align}
\end{prop}

\begin{proof}
Estimate \eqref{BSest4} is stated in \cite{FS} in the slightly
weaker form
\[
  \|u\|_{L^\infty(\Omega)} \,\le\, C \|\omega_\theta\|_{L^1(\Omega)}^{1/4}
  \|r^2 \omega_\theta\|_{L^1(\Omega)}^{1/4}\,\|\omega_\theta/r\|_{L^\infty(
  \Omega)}^{1/2}~,
\]
but the proof given there actually yields the stronger bound
\eqref{BSest4}, which can be compared with \eqref{BSest2} in the case
$p = 1$, $q = \infty$. To prove \eqref{BSest5}, we follow the same
approach as in \cite{FS}. Using translation and scaling invariance, it
is sufficient to show that the quantity
\begin{equation}\label{ur1}
  \frac{u_r}{r}\Big|_{r=1,z=0} \,=\, u_r(1,0) \,=\,
  \int_\Omega \frac{z}{\pi r^{1/2}} \,F'\left(\frac{(r-1)^2 +
  z^2}{r}\right)\omega_\theta(r,z)\dd r\dd z
\end{equation}
is bounded by $C \|\omega_\theta\|_{L^1(\Omega)}^{1/3} \,\|\omega_\theta/r
\|_{L^\infty(\Omega)}^{2/3}$. We decompose $\Omega = I_1 \cup I_2$, where
\[
  I_1 \,=\, \Bigl\{(r,z) \in \Omega \,\Big|\, \frac12 \le r \le 2\,,~
  -1 \le z \le 1\Bigr\}~, \qquad I_2 \,=\, \Omega\setminus I_1~.
\]
When integrating over the first region $I_1$, we use the fact
that $|F'(s)| \le Cs^{-1}$ and $r \approx 1$. We thus obtain
\[
  |u_r^{(1)}(1,0)| \,\le\, C \int_{I_1} \frac{|z|}{(r-1)^2 +
  z^2}\,|\omega_\theta(r,z)|\dd r\dd z \,\le\,
  C \int_{I_1} \frac{|\omega_\theta(r,z)|}{\bigl((r-1)^2 + z^2\bigr)^{1/2}}
  \dd r\dd z~.
\]
As in \cite{FS}, or in part ii) or Proposition~\ref{BSprop}, we
deduce that
\begin{equation}\label{ur2}
  |u_r^{(1)}(1,0)| \,\le\, C \|\omega_\theta\|_{L^1(I_1)}^{1/2}
  \|\omega_\theta\|_{L^\infty(I_1)}^{1/2} \,\le\, C \|\omega_\theta\|_{L^1(I_1)}^{1/3}
  \|\omega_\theta/r\|_{L^\infty(I_1)}^{2/3}~.
\end{equation}
In the complementary region, we use the optimal bound $|F'(s)| \le
Cs^{-5/2}$ and we observe that, if $(r,z) \in I_2$, then $(r-1)^2 + z^2
\ge C(r^2 + z^2)$ for some $C > 0$. We thus have
\[
  |u_r^{(2)}(1,0)| \,\le\, C \int_{I_2} \frac{|z|\,r^2}{\bigl((r-1)^2 +
  z^2\bigr)^{5/2}}\,|\omega_\theta(r,z)|\dd r\dd z \,\le\,
  C \int_{I_2} \frac{|\omega_\theta(r,z)|}{r^2 + z^2}\dd r\dd z~.
\]
Fix any $R > 0$ and denote $\Omega_R = \{(r,z) \in \Omega\,|\,
\rho \le R\}$, where $\rho = (r^2+z^2)^{1/2}$. Extending the
integration domain from $I_2$ to $\Omega = \Omega_R \cup
(\Omega\setminus\Omega_R)$, we compute
\begin{align*}
  |u_r^{(2)}(1,0)| \,&\le\, C \int_{\Omega_R} \frac{|\omega_\theta(r,z)|}{
  \rho^2}\dd r\dd z + C \int_{\Omega\setminus\Omega_R} \frac{|\omega_\theta(r,z)|}{
  \rho^2}\dd r\dd z \\
  \,&\le\, C \|\omega_\theta/r\|_{L^\infty(\Omega)} \int_{\Omega_R} \frac{1}{\rho}
  \dd r\dd z + C R^{-2} \int_{\Omega\setminus\Omega_R}|\omega_\theta(r,z)|
  \dd r\dd z \\
  \,&\le\,  C R \|\omega_\theta/r\|_{L^\infty(\Omega)} + C R^{-2}
  \|\omega_\theta\|_{L^1(\Omega)}~.
\end{align*}
Optimizing over $R > 0$ gives the bound  $|u_r^{(2)}(1,0)| \le
C\|\omega_\theta\|_{L^1(\Omega)}^{1/3} \,\|\omega_\theta/r\|_{L^\infty(\Omega)}^{2/3}$,
which together with \eqref{ur2} implies \eqref{BSest5}.
\end{proof}

\begin{rem}\label{alternative}
Estimate \eqref{BSest5} can also be obtained by observing that
\begin{equation}\label{lorentz}
  \Bigl\|\frac{u_r}{r}\Bigr\|_{L^\infty(\R^3)} \,\le\, C\,
  \Bigl\|\frac{\omega_\theta}{r}\Bigr\|_{L^{3,1}(\R^3)} \,\le\, C\,
  \Bigl\|\frac{\omega_\theta}{r}\Bigr\|_{L^1(\R^3)}^{1/3}\,
  \Bigl\|\frac{\omega_\theta}{r}\Bigr\|_{L^\infty(\R^3)}^{2/3}~,
\end{equation}
where $L^{3,1}(\R^3)$ denotes the Lorentz space. The first inequality
in \eqref{lorentz} is proved in \cite[Proposition~4.1]{AHK},
and the second one follows by real interpolation.
\end{rem}

\section{The semigroup associated with the linearized
equation}\label{sec3}

This section is devoted to the study of the linearized vorticity
equation \eqref{omeq}\:
\begin{equation}\label{omlin}
  \partial_t \omega_\theta \,=\, \Bigl(\partial_r^2 + \partial_z^2
  + \frac{1}{r}\partial_r - \frac{1}{r^2}\Bigr)\omega_\theta~,
\end{equation}
which is considered in the half-plane $\Omega = \{(r,z)
\,|\, r > 0,\, z \in \R\}$, with homogeneous Dirichlet condition
at the boundary $r = 0$. Given initial data $\omega_0 \in L^1(\Omega)$,
we denote by $\omega_\theta(t) = S(t)\omega_0$ the solution of
\eqref{omlin} at time $t > 0$.

\begin{lem}\label{Sformula}
For any $t > 0$, the evolution operator $S(t)$ associated with
Eq.~\eqref{omlin} is given by the explicit formula
\begin{equation}\label{Sdef}
  (S(t)\omega_0)(r,z) \,=\, \frac{1}{4\pi t}\int_\Omega \frac{\bar
  r^{1/2}}{r^{1/2}}\, H\Bigl(\frac{t}{r\bar r}\Bigr)\exp\Bigl(-
  \frac{(r-\bar r)^2 + (z-\bar z)^2}{4t}\Bigr)\omega_0(\bar r,\bar z)
  \dd\bar r\dd\bar z~,
\end{equation}
where the function $H : (0,+\infty) \to \R$ is defined by
\begin{equation}\label{Hdef}
  H(\tau) \,=\, \frac{1}{\sqrt{\pi\tau}}\int_{-\pi/2}^{\pi/2}
  e^{-\frac{\sin^2\phi}{\tau}}\cos(2\phi)\dd \phi~, \qquad
  \tau > 0~.
\end{equation}
\end{lem}

\begin{proof}
If $\omega_\theta$ is a solution of \eqref{omlin}, we observe
that the vector valued function $\omega = \omega_\theta e_\theta$
satisfies the usual heat equation $\partial_t \omega = \Delta
\omega$ in the whole space $\R^3$. For any $t > 0$, we thus have
the solution formula
\begin{equation}\label{heat3d}
  \omega(x,t) \,=\, \frac{1}{(4\pi t)^{3/2}}\int_{\R^3}
  e^{-\frac{|x-\bar x|^2}{4t}} \,\omega(\bar x,0)\dd\bar x~,
  \qquad x \in \R^3~.
\end{equation}
Denoting $x = (r\cos\theta,r\sin\theta,z)$ and $\bar x = (\bar r\cos
\bar\theta,\bar r\sin\bar\theta,\bar z)$, we can write \eqref{heat3d}
in the form
\begin{equation}\label{heat3d2}
  \omega_\theta(r,z,t)\begin{pmatrix}-\sin\theta \\ \cos\theta
  \\0\end{pmatrix} \,=\, \frac{1}{(4\pi t)^{3/2}}\int_0^\infty
  \!\!\int_\R \int_{-\pi}^\pi e^{-\frac{|x-\bar x|^2}{4t}}
  \,\omega_0(\bar r,\bar z)\begin{pmatrix}-\sin\bar\theta \\
  \cos\bar\theta \\0\end{pmatrix}\bar r\dd\bar\theta \dd\bar z\dd\bar r~,
\end{equation}
where
\[
  |x-\bar x|^2 \,=\, (r-\bar r)^2 + (z-\bar z)^2 + 4r\bar r
  \sin^2\frac{\theta-\bar\theta}{2}~.
\]
Now, if we integrate over the angle $\bar \theta$ in \eqref{heat3d2}
and use the definition \eqref{Hdef} of $H$, we see that
\eqref{heat3d2} is equivalent to $\omega_\theta(t) = S(t)\omega_0$
with $S(t)$ defined by \eqref{Sdef}.
\end{proof}

The function $H$ cannot be expressed in terms of elementary functions,
but all we need to know is the behavior of $H(\tau)$ as
$\tau \to 0$ and $\tau \to \infty$.

\begin{lem}\label{Hproperties}
The function $H : (0,\infty) \to \R$ defined by \eqref{Hdef}
is smooth and satisfies the asymptotic expansions\:
\\[2mm]
i) $\DS H(\tau) = 1 - \frac{3\tau}{4} + \cO(\tau^2)$ and
$H'(\tau) = -3/4 + \cO(\tau)$ as $\tau \to 0$; \\[2mm]
ii) $\DS H(\tau) = \frac{\pi^{1/2}}{4\tau^{3/2}} + \cO\Bigl(\frac{1}{
\tau^{5/2}}\Bigr)$ and $\DS H'(\tau) = -\frac{3\pi^{1/2}}{8\tau^{5/2}}
+ \cO\Bigl(\frac{1}{\tau^{7/2}}\Bigr)$ as $\tau \to \infty$.
\end{lem}

\begin{proof}
Expansion ii) follows immediately from \eqref{Hdef} if, in the 
expression $e^{-\frac{\sin^2\phi}{\tau}}$, we replace the exponential function 
by its Taylor series at the origin. Expansion i) can be deduced 
from the formula
\[
  H(\tau) \,=\, \frac1{\sqrt{\pi}}\int_{-\frac1{\sqrt\tau}}^{\frac1{\sqrt\tau}}
  \,e^{-x^2}\,\frac{1-2\tau x^2}{\sqrt{1-\tau x^2}}\,\dd x~,
\]
which is obtained by substituting $x=\frac{\sin\phi}{\sqrt{\tau}}$ in
the integral \eqref{Hdef}.
\end{proof}

\begin{rem}\label{Hbounds}
We believe that the function $H$ is decreasing, although we 
do not have a simple proof. In what follows, we only use the 
fact that the maps $\tau \mapsto \tau^\alpha H(\tau)$ and 
$\tau \mapsto \tau^\beta H'(\tau)$ are bounded if $0 \le \alpha 
\le 3/2$ and $0 \le \beta \le 5/2$.
\end{rem}

Let $(S(t))_{t\ge0}$ be the family of linear operators defined by
\eqref{Sdef} for $t>0$ and by $S(0) = \1$ (the identity operator).
By construction, we have the semigroup property\: $S(t_1+t_2) =
S(t_1)S(t_2)$ for all $t_1, t_2 \ge 0$. Further important
properties are collected in the following proposition.

\begin{prop}\label{Sprop}
The family $(S(t))_{t\ge0}$ defined by \eqref{Sdef} is a
strongly continuous semigroup of bounded linear operators
in $L^p(\Omega)$ for any $p \in [1,\infty)$. Moreover, if
$1 \le p \le q \le \infty$, the following estimates hold\:\\
i) If $\omega_0 \in L^p(\Omega)$, then
\begin{equation}\label{Sest1}
  \|S(t)\omega_0\|_{L^q(\Omega)} \,\le\, \frac{C}{t^{\frac1p-\frac1q}}
  \,\|\omega_0\|_{L^p(\Omega)}~, \qquad t > 0~.
\end{equation}
ii) If $f = (f_r,f_z) \in L^p(\Omega)^2$, then
\begin{equation}\label{Sest2}
  \|S(t)\div_* f\|_{L^q(\Omega)} \,\le\, \frac{C}{t^{\frac12+\frac1p
  -\frac1q}}\,\|f\|_{L^p(\Omega)}~, \qquad t > 0~,
\end{equation}
where $\div_* f = \partial_r f_r + \partial_z f_z$ denotes the
two-dimensional divergence of $f$.
\end{prop}

\begin{proof}
We claim that
\begin{equation}\label{point1}
  \frac{1}{4\pi t}\frac{\bar r^{1/2}}{r^{1/2}}\, H\Bigl(\frac{t}{r\bar r}
  \Bigr)\exp\Bigl(-\frac{(r-\bar r)^2 + (z-\bar z)^2}{4t}\Bigr) \,\le\,
  \frac{C}{t}\,\exp\Bigl(-\frac{(r-\bar r)^2 + (z-\bar z)^2}{5t}\Bigr)~,
\end{equation}
for all $(r,z) \in \Omega$, all $(\bar r,\bar z) \in \Omega$, and
all $t > 0$. Indeed, since $H$ is bounded, estimate \eqref{point1}
is obvious when $\bar r \le 2r$. If $\bar r > 2r$, we observe that
$\tau^{1/2}H(\tau)$ is bounded, so that
\[
  \frac{\bar r^{1/2}}{r^{1/2}}\, H\Bigl(\frac{t}{r\bar r}\Bigr) \,\le\,
  C\,\frac{\bar r^{1/2}}{r^{1/2}}\,\frac{r^{1/2}\,\bar r^{1/2}}{t^{1/2}}
  \,=\, C\,\frac{\bar r}{t^{1/2}} \,\le\, C\,\frac{|r-\bar r|}{t^{1/2}}~,
\]
where in the last inequality we used the fact that $\bar r < 2(\bar
r-r) = 2|r - \bar r|$. As $xe^{-x^2/4} \le C e^{-x^2/5}$ for all $x
\ge 0$, we conclude that \eqref{point1} holds in all cases. This
provides for the integral kernel in \eqref{Sdef} a pointwise upper
bound in terms of the usual heat kernel in the whole plane $\R^2$,
with a diffusion coefficient equal to $5/4$ instead of $1$. Thus 
estimate \eqref{Sest1} follows immediately from Young's inequality, 
as in the 2D case.

To prove \eqref{Sest2}, we assume that $\omega_0 = \div_* f =
\partial_r f_r + \partial_z f_z$, and we integrate by parts
in \eqref{Sdef} to obtain the identity
\[
  (S(t)\div_* f)(r,z) \,=\, \frac{1}{4\pi t}\int_\Omega
  \frac{\bar r^{1/2}}{r^{1/2}}\,\exp\Bigl(-\frac{(r-\bar r)^2 +
  (z-\bar z)^2}{4t}\Bigr)(A_r f_r + A_z f_z)\dd\bar r\dd\bar z~,
\]
where
\[
  A_r \,=\, \frac{t}{r \bar r^2}\,H'\Bigl(\frac{t}{r \bar r}\Bigr)
  - \Bigl(\frac{1}{2\bar r} + \frac{r-\bar r}{2t}\Bigr)
  H\Bigl(\frac{t}{r \bar r}\Bigr)~, \qquad A_z \,=\, -\frac{z-\bar z}{2t}
  \,H\Bigl(\frac{t}{r \bar r}\Bigr)~.
\]
Proceeding as above and using Remark~\ref{Hbounds}, it is
straightforward to verify that
\begin{equation}\label{point2}
  \frac{1}{4\pi t} \frac{\bar r^{1/2}}{r^{1/2}}\,\exp\Bigl(
  -\frac{(r{-}\bar r)^2 + (z{-}\bar z)^2}{4t}\Bigr)\Bigl(|A_r|+
  |A_z|\Bigr)\,\le\, \frac{C}{t^{3/2}}\,\exp\Bigl(-\frac{(r{-}\bar r)^2
  + (z{-}\bar z)^2}{5t}\Bigr)~,
\end{equation}
for all $(r,z) \in \Omega$, all $(\bar r,\bar z) \in \Omega$,
and all $t > 0$. Thus estimate \eqref{Sest2} follows again from
Young's inequality, as in the 2D case.

Finally, we show that the semigroup $(S(t))_{t\ge0}$ is strongly
continuous in $L^p(\Omega)$ if $1 \le p < \infty$.  All we need to
verify is the continuity at the origin. Given $\omega_0 \in
L^p(\Omega)$, we denote by $\overline{\omega}_0 : \R^2 \to \R$ the
function obtained by extending $\omega_0$ by zero outside $\Omega$.
Using the change of variables $\bar r = r + \sqrt{t}\rho$,
$\bar z = z + \sqrt{t}\zeta$ in \eqref{Sdef}, we obtain
the identity
\[
  \Bigl(S(t)\omega_0 - \omega_0\Bigr)(r,z) \,=\,
  \frac{1}{4\pi}\int_{\R^2} e^{-\frac{\rho^2+\zeta^2}{4}}
  \,\Psi(r,z,\rho,\zeta,t) \dd\rho\dd\zeta~,
\]
for all $(r,z) \in \Omega$, where
\[
  \Psi(r,z,\rho,\zeta,t) \,=\, \Bigl(1+\frac{\sqrt{t}\rho}{r}\Bigr)^{1/2}
  H\Bigl(\frac{t}{r(r+\sqrt{t}\rho)}\Bigr)\,\overline{\omega}_0(r
  + \sqrt{t}\rho,z + \sqrt{t}\zeta) - \overline{\omega}_0(r,z)~.
\]
By Minkowski's integral inequality, we deduce that
\begin{equation}\label{mink}
  \|S(t)\omega_0 - \omega_0\|_{L^p(\Omega)} \,\le\,  \frac{1}{4\pi}
  \int_{\R^2} e^{-\frac{\rho^2+\zeta^2}{4}}\,\|\Psi(\cdot,\cdot,\rho,\zeta,t)
  \|_{L^p(\Omega)} \dd\rho\dd\zeta~.
\end{equation}
Now, the estimates we used in the proof of \eqref{point1} show that
\begin{equation}\label{point3}
  \Bigl(1+\frac{\sqrt{t}\rho}{r}\Bigr)^{1/2} H\Bigl(\frac{t}{r(r+\sqrt{t}
  \rho)}\Bigr) \,\le\, C(1+|\rho|)~,
\end{equation}
whenever $r > 0$ and $r + \sqrt{t}\rho > 0$. This immediately implies
that
\[
  \|\Psi(\cdot,\cdot,\rho,\zeta,t)\|_{L^p(\Omega)} \,\le\, C(1+|\rho|)
  \|\omega_0\|_{L^p(\Omega)}~,
\]
for all $(\rho,\zeta) \in \R^2$ and all $t > 0$. Since the left-hand
side of \eqref{point3} converges to $1$ as $t \to 0$, and since
translations act continuously in $L^p(\R^2)$ for $p < \infty$, it also
follows from estimate \eqref{point3} and Lebesgue's dominated convergence
theorem that
\[
  \|\Psi(\cdot,\cdot,\rho,\zeta,t)\|_{L^p(\Omega)} \,\longrightarrow\,
  0 \quad \hbox{as }t \to 0~,
\]
for all $(\rho,\zeta) \in \R^2$. Thus another application of
Lebesgue's theorem implies that the right-hand side of \eqref{mink}
converges to zero as $t \to 0$, which is the desired result.
\end{proof}

As in the previous section, it is possible to derive weighted
estimates for the linear semigroup \eqref{Sdef}. The proof of the
following result is very similar to that of Propositions~\ref{BSprop2}
and \ref{Sprop}, and can thus be left to the reader.

\begin{prop}\label{Sprop2}
Let $1 \le p \le q \le \infty$ and $-1 \le \alpha \le \beta \le 2$.
If $r^\beta\omega_0 \in L^p(\Omega)$, then
\begin{equation}\label{Sweight1}
  \|r^\alpha S(t)\omega_0\|_{L^q(\Omega)} \,\le\, \frac{C}{t^{\frac1p-\frac1q
  + \frac{\beta-\alpha}{2}}} \,\|r^\beta \omega_0\|_{L^p(\Omega)}~, \qquad t > 0~.
\end{equation}
Moreover, if $-1 \le \alpha \le \beta \le 1$ and $r^\beta f \in L^p(\Omega)^2$,
then
\begin{equation}\label{Sweight2}
  \|r^\alpha S(t)\div_* f\|_{L^q(\Omega)} \,\le\, \frac{C}{t^{\frac12+\frac1p
  -\frac1q+\frac{\beta-\alpha}{2}}}\,\|r^\beta f\|_{L^p(\Omega)}~, \qquad t > 0~.
\end{equation}
\end{prop}

\section{Local existence of solutions}\label{sec4}

Equipped with the results of the previous sections, we now
take up the proof of Theorems~\ref{main1} and \ref{main2}.
In view of the divergence-free condition in \eqref{diffBS},
the evolution equation \eqref{omeq} for the axisymmetric
vorticity $\omega_\theta$ can be written in the equivalent form
\begin{equation}\label{omeq2}
  \partial_t \omega_\theta + \div_*(u\,\omega_\theta)
  \,=\, \Bigl(\partial_r^2 + \partial_z^2 + \frac{1}{r}\partial_r
  - \frac{1}{r^2}\Bigr)\omega_\theta~,
\end{equation}
where $\div_*(u\,\omega_\theta) = \partial_r (u_r \omega_\theta) +
\partial_z (u_z \omega_\theta)$. Given initial data $\omega_0$,
the integral equation associated with \eqref{omeq2} is
\begin{equation}\label{omint}
  \omega_\theta(t) \,=\, S(t)\omega_0 - \int_0^t S(t-s)
  \div_*(u(s)\omega_\theta(s))\dd s~, \qquad t > 0~,
\end{equation}
where $S(t)$ denotes the linear semigroup defined in \eqref{Sdef}. In
this section, our goal is to prove local existence and uniqueness of
solutions to \eqref{omint} using a standard fixed point argument, in
the spirit of Kato \cite{Ka1}. For the sake of clarity, we first treat
the case where $\omega_0 \in L^1(\Omega)$, and then consider the more
complicated situation where $\omega_0$ is a finite measure with
sufficiently small atomic part. This will establish the local
well-posedness claims in Theorems~\ref{main1} and \ref{main2},
respectively.

\subsection{Local existence when the initial vorticity
is integrable}
\label{sec41}

\begin{prop}\label{propex1}
For any initial data $\omega_0 \in L^1(\Omega)$, there exists
$T = T(\omega_0) > 0$ such that the integral equation \eqref{omint}
has a unique solution
\begin{equation}\label{ompropT}
  \omega_\theta \in C^0([0,T],L^1(\Omega)) \cap C^0((0,T],L^\infty(\Omega))~.
\end{equation}
Moreover $\|\omega_\theta(t)\|_{L^1(\Omega)} \le \|\omega_0\|_{L^1(\Omega)}$
for all $t \in [0,T]$, and estimate \eqref{asym1} holds. Finally,
if $\|\omega_0\|_{L^1(\Omega)}$ is small enough, the local existence
time  $T > 0$ can be taken arbitrarily large.
\end{prop}

\begin{proof}
We follow the same approach as in the two-dimensional case, see
e.g. \cite{BA,Ga,Ka2}. Given $T > 0$ we introduce the function space
\begin{equation}\label{XTdef}
  X_T \,=\, \Bigl\{\omega_\theta \in C^0((0,T],L^{4/3}(\Omega)\,\Big|\,
  \|\omega_\theta\|_{X_T} < \infty\Bigr\}~,
\end{equation}
equipped with the norm
\[
  \|\omega_\theta\|_{X_T} \,=\, \sup_{0 < t \le T}t^{1/4}\|\omega_\theta(t)
  \|_{L^{4/3}(\Omega)}~.
\]
For any $t \ge 0$, we denote $\omega_\lin(t) = S(t)\omega_0$,
where $S(t)$ is the linear semigroup \eqref{Sdef}. It then
follows from Proposition~\ref{Sprop} that $\omega_\lin \in
X_T$ for any $T > 0$. For later use, we define
\begin{equation}\label{C1def}
  C_1(\omega_0,T) \,=\, \|\omega_\lin\|_{X_T} \,=\,
  \sup_{0 < t \le T} t^{1/4}\|S(t)\omega_0\|_{L^{4/3}(\Omega)}~.
\end{equation}
In view of \eqref{Sest1}, there exists a universal constant $C_2 > 0$
such that $C_1(\omega_0,T) \le C_2 \|\omega_0\|_{L^1(\Omega)}$ for any
$T > 0$. Moreover, since $L^1(\Omega) \cap L^{4/3}(\Omega)$ is dense
in $L^1(\Omega)$, it also follows from \eqref{Sest1} that
$C_1(\omega_0,T) \to 0$ as $T \to 0$, for any $\omega_0 \in
L^1(\Omega)$.

Given $\omega_\theta \in X_T$ and $p \in [1,2)$, we define a map 
$\cF \omega_\theta : (0,T] \to L^p(\Omega)$ in the following way\:
\begin{equation}\label{cFdef}
  (\cF \omega_\theta)(t) \,=\, \int_0^t S(t-s)\div_*(u(s)
  \omega_\theta(s))\dd s~, \qquad 0 < t \le T~,
\end{equation}
where it is understood that $u(s)$ is the velocity field obtained from
$\omega_\theta(s)$ via the axisymmetric Biot-Savart law \eqref{BSu}. 
Using estimate \eqref{Sest2}, H\"older's inequality, and the bound 
\eqref{BSest1}, we obtain for $t \in (0,T]$\:
\begin{align}\nonumber
  t^{1-\frac1p}\|(\cF \omega_\theta)(t)\|_{L^p(\Omega)} \,&\le\,
  t^{1-\frac1p}\int_0^t \frac{C}{(t-s)^{\frac32-\frac1p}}\,
  \|u(s)\omega_\theta(s)\|_{L^1(\Omega)}\dd s \\ \nonumber
  \,&\le\, t^{1-\frac1p}\int_0^t \, \frac{C}{(t-s)^{\frac32-\frac1p}} 
  \|u(s)\|_{L^4(\Omega)}\|\omega_\theta(s)\|_{L^{4/3}(\Omega)}\dd s \\ \label{cFest}
  \,&\le\, t^{1-\frac1p}\int_0^t \frac{C}{(t-s)^{\frac32-\frac1p}}
  \|\omega_\theta(s)\|_{L^{4/3}(\Omega)}^2\dd s \\ \nonumber
  \,&\le\, t^{1-\frac1p}\int_0^t \frac{C}{(t-s)^{\frac32-\frac1p}}\,\frac{
  \|\omega_\theta\|_{X_T}^2}{s^{\frac12}}\dd s \,\le\, C \|\omega_\theta\|_{X_T}^2~.
\end{align}
It is also straightforward to verify that the quantity $(\cF
\omega_\theta)(t)$ depends continuously on the time parameter $t \in
(0,T]$ in the topology of $L^p(\Omega)$. Choosing $p = 4/3$, we deduce
that $\cF \omega_\theta \in X_T$ and $\|\cF\omega_\theta\|_{X_T} \le
C_3 \|\omega_\theta\|_{X_T}^2$ for some $C_3 > 0$. More generally, we
have the Lipschitz estimate
\begin{equation}\label{cFLip}
  \|\cF \omega_\theta - \cF \tilde \omega_\theta\|_{X_T} \,\le\, C_3
 \bigl(\|\omega_\theta\|_{X_T} + \|\tilde \omega_\theta\|_{X_T}\bigr)
  \|\omega_\theta- \tilde\omega_\theta\|_{X_T}~,
\end{equation}
for all $\omega_\theta, \tilde\omega_\theta \in X_T$.

Now we consider the map $\cG : X_T \to X_T$ defined by $\cG
\omega_\theta = \omega_\lin - \cF \omega_\theta$. We fix $R > 0$ such
that $2 C_3 R < 1$, and denote by $B_R$ the closed ball of radius $R$
centered at the origin in $X_T$. If $C_1(\omega_0,T) \le R/2$, the
estimates above show that $\cG$ maps $B_R$ into $B_R$ and is a strict
contraction there, so that (by the Banach fixed point theorem) $\cG$
has a unique fixed point $\omega_\theta$ in $B_R$. By construction,
$\omega_\theta$ is a solution to the integral equation \eqref{omint}
in $X_T$. The condition $C_1(\omega_0,T) \le R/2$ can be fulfilled in
two different ways. If the initial data are small enough so that $C_2
\|\omega_0\|_{L^1(\Omega)} \le R/2$, the existence time $T > 0$ can be
chosen arbitrarily, and the fixed point argument therefore establishes
{\em global existence for small data in $L^1(\Omega)$}.  On the other
hand, for larger initial data $\omega_0$, we can always choose $T > 0$
small enough so that $C_1(\omega_0,T) \le R/2$, hence we also have
{\em local existence for arbitrary data}.

\begin{rems}\label{Texist}~\\
{\bf 1.} For large data, the local existence time $T > 0$ given by
the fixed point argument depends on the initial data, and it is not
possible to bound $T$ from below using the norm $\|\omega_0\|_{L^1(\Omega)}$
only. However, if $\omega_0 \in L^1(\Omega) \cap L^p(\Omega)$ for some
$p > 1$, then (by Proposition~\ref{Sprop}) an upper bound on
$\|\omega_0\|_{L^p(\Omega)}$ provides a lower bound
on the local existence time $T$.\\[1mm]
{\bf 2.}  For later use, we note that the fixed point argument also
proves that the solution $\omega_\theta$ depends continuously on the
initial data. More precisely, if $K \subset L^1(\Omega)$ is any
compact set, or any sufficiently small neighborhood of a given point,
we can take the same local existence time $T > 0$ for all initial data
$\omega_0 \in K$, and there exists $C > 0$ such that
\begin{equation}\label{omcont}
  \sup_{t \in [0,T]} \|\omega_\theta(t) - \tilde\omega_\theta(t)
  \|_{L^1(\Omega)} + \|\omega_\theta - \tilde\omega_\theta\|_{X_T} \,\le\,
  C\|\omega_0 - \tilde\omega_0\|_{L^1(\Omega)}~,
\end{equation}
for all $\omega_0, \tilde \omega_0 \in K$, where $\omega_\theta,
\tilde\omega_\theta \in X_T$ denote the solutions corresponding to
the initial data $\omega_0, \tilde \omega_0$, respectively.
\end{rems}

To conclude the proof of Proposition~\ref{propex1}, we establish
a few additional properties of the local solution $\omega_\theta \in
X_T$. We first note that
\begin{equation}\label{zerotime}
  \lim_{t \to 0} t^{1/4}\|\omega_\theta(t)\|_{L^{4/3}(\Omega)}
  \,=\, \lim_{T \to 0} \|\omega_\theta\|_{X_T} \,=\, 0~,
\end{equation}
because, when $T > 0$ is small, the fixed point argument holds in the
ball $B_R$ with $R = 2C_1(\omega_0,T)$. This proves \eqref{asym1} for
$p = 4/3$. Next, using \eqref{cFest} with $p = 1$, we see that 
that the map $t \mapsto (\cF \omega_\theta)(t)$ is continuous in the
topology of $L^1(\Omega)$ and satisfies $\|(\cF \omega_\theta)(t)
\|_{L^1(\Omega)} \le C \|\omega_\theta\|_{X_T}^2$ for all $t \in
(0,T]$.  In view of \eqref{zerotime}, this implies that the map
$\omega_\theta - \omega_\lin \equiv - \cF \omega_\theta$ belongs to
$C^0([0,T], L^1(\Omega)$ and vanishes at $t = 0$. In particular, using
Proposition~\ref{Sprop}, we conclude that $\omega_\theta \in
C^0([0,T],L^1(\Omega))$. That $\|\omega_\theta(t)\|_{L^1(\Omega)}$ is
a non-increasing function of time is a well known fact, which will be
discussed in Lemma~\ref{L1lem} below. Finally, to prove that
$\omega_\theta \in C^0((0,T],L^p(\Omega))$ for any $p \in (1,\infty]$
and that \eqref{asym1} holds, we use a standard bootstrap argument
which we explain in some detail because it will be used again in
Section~\ref{sec5}. For $p \in (1,\infty]$ we define
\begin{equation}\label{MNdef}
  M_p(T) \,=\, \sup_{0 < t \le T} t^{1-\frac1p} \|S(t)\omega_0\|_{L^p
  (\Omega)}~, \qquad \hbox{and}\quad
  N_p(T) \,=\, \sup_{0 < t \le T} t^{1-\frac1p} \|\omega_\theta(t)\|_{L^p
  (\Omega)}~.
\end{equation}
Then $M_p(T) \to 0$ as $T \to 0$, and we already know that $N_p(T) \to
0$ as $T \to 0$ for $p = 4/3$, hence for all $p \in (1,4/3]$ by
interpolation. To prove the same result for $p > 4/3$, we split the
integral in \eqref{omint} in two parts and estimate the nonlinear term
as follows\:
\[
  \|u\,\omega_\theta\|_{L^r(\Omega)} \,\le\, C \|u\|_{L^s(\Omega)}
  \|\omega_\theta\|_{L^{q_2}(\Omega)} \,\le\, C \|\omega_\theta\|_{L^{q_1}(\Omega)}
   \|\omega_\theta\|_{L^{q_2}(\Omega)}~, \qquad s \,=\, \frac{2q_1}{2-q_1}~,
\]
where $\frac43 \le q_1 < 2$, $\frac43 \le q_2 \le \infty$, and $\frac1r =
\frac{1}{q_1} + \frac{1}{q_2} - \frac12$. We thus obtain
\begin{align}\nonumber
  \|\omega_\theta(t)\|_{L^p(\Omega)} \,\le\, \|S(t)\omega_0\|_{L^p
  (\Omega)} &+ C\int_0^{t/2} \frac{\|\omega_\theta(s)\|_{L^q(\Omega)}^2}{(t-s)^{
  \frac2q -\frac1p}}\dd s \\ \label{intsplit}
  &+ C\int_{t/2}^t \!\frac{\|\omega_\theta(s)
  \|_{L^{q_1}(\Omega)} \|\omega_\theta(s)\|_{L^{q_2}(\Omega)}}{(t-s)^{\frac{1}{q_1}
  +\frac{1}{q_2} -\frac1p}}\dd s~,
\end{align}
where the exponents $p \in [1,\infty]$, $q,q_1 \in [4/3,2)$ and
$q_2 \in [4/3,\infty]$ are assumed to satisfy
\begin{equation}\label{MNexp}
  \frac12 \,\le\, \frac2q - \frac1p~, \qquad \hbox{and}\quad
  \frac12 \,\le\, \frac1{q_1} + \frac1{q_2} - \frac1p \,<\, 1~.
\end{equation}
Multiplying both sides of \eqref{intsplit} by $t^{1-\frac1p}$ and
taking the supremum over $t \in (0,T]$ we obtain the useful bound
\begin{equation}\label{MNrec}
  N_p(T) \,\le\, M_p(T) + C_{p,q}N_q(T)^2 + C_{p,q_1,q_2}N_{q_1}(T)
  N_{q_2}(T)~,
\end{equation}
where $C_{p,q}$ and $C_{p,q_1,q_2}$ are positive constants. If
we choose $q = q_1 = q_2 = 4/3$, we deduce from \eqref{MNrec}
that $N_p(T) \to 0$ as $T \to 0$ for any $p < 2$. Then,
taking $q = 4/3$ and $q_1 = q_2$ sufficiently close to $2$,
we obtain the same result for any $p < \infty$. Finally,
choosing $q = 4/3$, $q_1 = 3/2$, and $q_2 = 4$, we conclude
from \eqref{MNrec} that $N_\infty(T) \to 0$ as $T \to 0$,
which proves \eqref{asym1}.

Our final task is to discuss the uniqueness of the solution to
\eqref{omint}. We first observe that, given $\omega_0 \in
L^1(\Omega)$, the solution $\omega_\theta \in X_T$ given by the fixed
point argument is, by construction, the only solution of \eqref{omint}
in $X_T$ satisfying \eqref{zerotime}. In fact, using a nice argument
due to Brezis \cite{Br}, it is possible to prove uniqueness in a much
larger class. Indeed, given any $T > 0$, assume that $\omega_\theta
\in C^0([0,T],L^1(\Omega)) \cap C^0((0,T],L^\infty (\Omega))$ is a
mild solution of \eqref{omeq2} on $(0,T]$ in the following sense\:
\begin{equation}\label{omint2}
  \omega_\theta(t) \,=\, S(t-t_0)\omega_\theta(t_0) - \int_{t_0}^t S(t-s)
  \div_*(u(s)\omega_\theta(s))\dd s~, \qquad 0 < t_0 \le t \le T~.
\end{equation}
The set $K =\{\omega_\theta(t)\,|\, t \in [0,T]\}$ is compact in
$L^1(\Omega)$, hence the fixed point argument allows us to construct a
local solution in $X_{\tilde T}$ for all initial data $\tilde \omega_0
\in K$, with a common existence time $\tilde T > 0$ (without loss of
generality, we assume henceforth that $\tilde T \le T/2$).  That
solution is denoted by $\tilde \omega_\theta(t) = \Sigma(t) \tilde
\omega_0$ for $t \in [0,\tilde T]$. By the observation above, for all
$t_0 \in (0,\tilde T]$ we have the relation
\begin{equation}\label{omrel}
  \omega_\theta(t) \,=\, \Sigma(t-t_0)\omega_\theta(t_0)~,
  \qquad t \in [t_0,t_0+\tilde T]~,
\end{equation}
because the left-hand side is a solution of \eqref{omint2} on 
the time interval $[t_0,t_0+\tilde T]$ and we obviously have
$(t-t_0)^{1/4}\|\omega_\theta(t)\|_{L^{4/3}(\Omega)} \to 0$ as $t \to t_0$, 
which is the analogue of condition \eqref{zerotime}. Now, for
any fixed $t \in (0,\tilde T]$, it follows from \eqref{omcont} that
\[
  \|\Sigma(t-t_0)\omega_\theta(t_0) - \Sigma(t-t_0)\omega_\theta(0)\|_{L^1(\Omega)}
  \,\le\, C \|\omega_\theta(t_0) - \omega_\theta(0)\|_{L^1(\Omega)}
  \,\longrightarrow\, 0~,
\]
as $t_0 \to 0$, and it is also clear that $\|\Sigma(t-t_0)\omega_\theta(0)-
\Sigma(t)\omega_\theta(0)\|_{L^1(\Omega)} \to 0$ as $t_0 \to 0$. Thus taking
the limit $t_0 \to 0$ in \eqref{omrel} we obtain the relation
$\omega_\theta(t) = \Sigma(t)\omega_\theta(0)$ for $t \in [0,\tilde T]$, which
means that the solution $\omega_\theta$ we started with coincides 
on the time interval $[0,\tilde T]$ with the solution constructed 
from the initial data $\omega_\theta(0)$ by the fixed point argument.
\end{proof}

\subsection{The case where the initial vorticity is a finite
measure}\label{sec42}

We next consider the more general case where the initial vorticity
$\omega_0$ in \eqref{omint} is a finite measure on $\Omega$, which is
no longer absolutely continuous with respect to the Lebesgue measure
$\dd r\dd z$. For convenience we denote $\mu = \omega_0$, and we
recall the canonical decomposition $\mu = \mu_{ac} + \mu_{sc} +
\mu_{pp}$ where $\mu_{ac}$ is absolutely continuous with respect to
Lebesgue's measure, $\mu_{pp}$ is a (countable) collection of point
masses, and $\mu_{sc}$ is the ``singularly continuous'' part which has
no atoms yet is supported on a set of zero Lebesgue measure. We have
$\mu_{ac} \perp \mu_{sc} \perp \mu_{pp}$, which means that the three
measures are mutually singular. In particular, the total variation
norm of $\mu$ satisfies
\[
  \|\mu\|_\tv \,=\, \|\mu_{ac}\|_\tv +  \|\mu_{sc}\|_\tv +  \|\mu_{pp}\|_\tv~.
\]
The linear semigroup $S(t)$ acts on the measure $\mu$ by the
formula
\begin{equation}\label{Sdef2}
  (S(t)\mu)(r,z) \,=\, \frac{1}{4\pi t}\int_\Omega \frac{\bar
  r^{1/2}}{r^{1/2}}\, H\Bigl(\frac{t}{r\bar r}\Bigr)\exp\Bigl(-
  \frac{(r-\bar r)^2 + (z-\bar z)^2}{4t}\Bigr)\dd\mu(\bar r,\bar z)~,
\end{equation}
which generalizes \eqref{Sdef}, and we have the following estimates\:

\begin{prop}\label{Sprop3}
Let $\mu$ be a finite measure on $\Omega$. Then
\begin{equation}\label{Sest3}
  \sup_{t > 0} t^{1-\frac1p}\,\|S(t)\mu\|_{L^p(\Omega)} \,\le\,
  C \|\mu\|_\tv~, \qquad 1 \le p \le \infty~,
\end{equation}
and
\begin{equation}\label{Sest4}
   L_p(\mu) \,:=\, \limsup_{t \to 0}\,t^{1-\frac1p}\|S(t)\mu\|_{L^p(\Omega)}
   \,\le\, C \|\mu_{pp}\|_\tv~, \qquad 1 < p \le \infty~.
\end{equation}
\end{prop}

\begin{proof}
Estimate \eqref{Sest3} can be established as in Proposition~\ref{Sprop},
using the pointwise upper bound \eqref{point1}. To prove \eqref{Sest4}
we proceed as in the two-dimensional case \cite{Ga,GMO}, with
minor modifications. We know from \eqref{Sest3} that $L_p(\mu) \le
C \|\mu\|_\tv$, hence using the canonical decomposition we find
\[
  L_p(\mu) \,\le\,  L_p(\mu_{ac}) + L_p(\mu_{sc}) + L_p(\mu_{pp})
  \,\le\, L_p(\mu_{ac}) + L_p(\mu_{sc}) + C \|\mu_{pp}\|_\tv~.
\]
Therefore, we only need to show that $L_p(\mu_{ac})  = L_p(\mu_{sc})
= 0$. In fact, it is sufficient to prove that for $p = \infty$,
because the result then follows for $1 < p < \infty$ by interpolation.  

From now on, we thus assume that $\mu$ is a non-atomic finite measure
on $\Omega$, and we denote by $|\mu|$ the positive measure which
represents the total variation of $\mu$. Given any point $\xi = (r,z)
\in \R^2$ and any radius $\delta > 0$, we define
\[
  B(\xi,\delta) \,=\, \bigl\{\bar\xi \in \Omega\,\big|\,
  |\xi - \bar\xi| \le \delta\bigr\}~,
\]
where $|\xi - \bar\xi|$ is the Euclidean distance between $\xi$ and
$\bar \xi$. We claim that, for any $\epsilon > 0$, there exists
$\delta > 0$ such that
\begin{equation}\label{abscont}
  \sup_{\xi \in \Omega}|\mu|(B(\xi,\delta)) \,\le\, \epsilon~.
\end{equation}
Indeed, if that property fails, there exist $\epsilon > 0$, a sequence
$(\xi_n)$ of points of $\Omega$, and a sequence $(\delta_n)$ of
positive real numbers such that $\delta_n \to 0$ as $n \to \infty$ and
$|\mu|(B(\xi_n,\delta_n)) > \epsilon$ for all $n \in \N$. It is clear
that the sequence $(\xi_n)$ is bounded, because $|\mu|$ is a finite
measure. Thus, after extracting a subsequence, we can assume that
$\xi_n$ converges as $n \to \infty$ to some point $\bar \xi = (\bar
r,\bar z)\in \bar \Omega$. For any $\delta > 0$ we thus have
$|\mu|(B(\bar\xi,\delta)) > \epsilon$, since $B(\bar\xi,\delta)
\supset B(\xi_n,\delta_n)$ when $n$ is sufficiently large. We conclude
that
\[
  |\mu|\Bigl(\bigcap_{\delta > 0} B(\bar\xi,\delta)\Bigr)
  \,\ge\, \epsilon \,>\,0~.
\]
But this is impossible, because the intersection above is empty if
$\bar\xi \in \partial\Omega$ and equal to the singleton $\{\bar\xi\}$
if $\bar\xi \in \Omega$, and we assumed that the measure $\mu$ is
non-atomic. Hence property \eqref{abscont}, which can be interpreted
as a weak form of absolute continuity with respect to Lebesgue's
measure, must hold.

Now, for any given $t > 0$, there exists $\bar\xi(t) \in \Omega$ such
that
\[
  |(S(t)\mu)(\bar\xi(t))| \,=\, \|S(t)\mu\|_{L^\infty(\Omega)}~,
\]
because the map $\xi \mapsto (S(t)\mu)(\xi)$ is continuous and vanishes at
infinity as well as on the boundary $\partial\Omega$. Using definition
\eqref{Sdef2} and the pointwise estimate \eqref{point1}, we thus obtain
\[
  t \|S(t)\mu\|_{L^\infty(\Omega)} \,\le\, C\int_{B(\bar\xi(t),\delta)}
  e^{-\frac{|\xi-\bar\xi(t)|^2}{5t}}\dd|\mu|(\xi) + C\int_{\Omega\setminus
  B(\bar\xi(t),\delta)} e^{-\frac{|\xi-\bar\xi(t)|^2}{5t}}\dd|\mu|(\xi)~,
\]
where $\epsilon$ and $\delta$ are as in \eqref{abscont}. The first
integral is bounded by $C |\mu| (B(\bar\xi(t),\delta)) \le C\epsilon$, and
the second one by $C e^{-\delta^2/(5t)}|\mu|(\Omega)$. It follows that
\[
  L_\infty(\mu) \,=\, \limsup_{t \to 0}\,t\|S(t)\mu\|_{L^\infty(\Omega)}
  \,\le\, C\epsilon~,
\]
and since $\epsilon > 0$ was arbitrary we conclude that
$L_\infty(\mu) = 0$, which is the desired result.
\end{proof}

\begin{prop}\label{propex2}
There exist positive constants $\epsilon$ and $C$ such that,
for any initial data $\omega_0 \in \cM(\Omega)$ with
$\|(\omega_0)_\pp\|_\tv \le \epsilon$, one can choose $T = T(\omega_0)
> 0$ such that the integral equation \eqref{omint2} has a unique
solution $\omega_\theta \in C^0((0,T],L^1(\Omega) \cap L^\infty(\Omega))$
satisfying \eqref{main2prop} and such that $\omega_\theta(t)
\weakto \omega_0$ as $t \to 0$. Moreover, if $\|\omega_0\|_\tv$ is
small enough, the local existence time  $T > 0$ can be taken
arbitrarily large.
\end{prop}

\begin{proof}
We briefly indicate how the proof of Proposition~\ref{propex1} has to
be modified to handle the case where $\mu = \omega_0 \in \cM(\Omega)$.
We use exactly the same function space $X_T$ defined in \eqref{XTdef},
and observe that the fixed point argument works in the ball $B_R
\subset X_T$ provided $C_1(\mu,T) \le R/2$, where $C_1(\mu,T)$ is
defined as in \eqref{C1def} and $R > 0$ satisfies $2C_3 R < 1$ with
$C_3$ as in \eqref{cFLip}. From \eqref{Sest3} we know that $C_1(\mu,T)
\le C_2 \|\mu\|_\tv$ for any $T > 0$, hence we again obtain global
existence and uniqueness in $B_R$ if the initial vorticity is small
enough so that $C_2 \|\mu\|_\tv \le R/2$.  For larger data, we can use
\eqref{Sest4} which gives
\[
  \lim_{T \to 0} C_1(\mu,T) \,=\, L_{4/3}(\mu) \,\le\,
  C_4\|\mu_{pp}\|_\tv~,
\]
for some positive constant $C_4$. Thus, if the atomic part of the
initial vorticity is small enough so that $C_4\|\mu_{pp}\|_\tv < R/2$,
we can take $T > 0$ such that $C_1(\mu,T) \le R/2$, and the fixed
point argument proves local existence and uniqueness in $B_R$.
In contrast, if $4C_3 C_4\|\mu_{pp}\|_\tv \ge 1$, it is impossible
to choose $R > 0$ and $T > 0$ so that the fixed point argument works
in the ball $B_R \subset X_T$, and the method above completely
fails.

Assuming that the fixed point argument works in the ball
$B_R \subset X_T$, we can establish some additional properties
of the solution $\omega_\theta\in B_R$ as in the case of
integrable initial data. For instance, it is straightforward to
verify that $\omega_\theta - \omega_\lin \in C^0((0,T),L^1(\Omega)
\cap L^\infty(\Omega))$, where $\omega_\lin(t) = S(t)\mu$.
However property \eqref{zerotime} fails if $\mu_{pp}
\neq 0$, and we cannot argue as in Section~\ref{sec41} to
show that $\|\omega_\theta(t) - \omega_\lin(t)\|_{L^1(\Omega)}$
converges to zero as $t \to 0$. To prove that, we first define
\[
  \delta \,=\, \limsup_{t \to 0}t^{1/4} \|\omega_\theta(t) -
  \omega_\lin(t)\|_{L^{4/3}(\Omega)} \,=\, \limsup_{T \to 0}
  \|\omega_\theta - \omega_\lin\|_{X_T}~.
\]
Since $\omega_\theta - \omega_\lin = (\cF\omega_\lin -
\cF\omega_\theta) - \cF\omega_\lin$ and $\|\omega_\lin\|_{X_T} +
\|\omega_\theta\|_{X_T} \le 2R$, we can use \eqref{cFLip} to
obtain the estimate $\delta \le 2C_3R\delta + \ell_{4/3}(\mu)$,
where
\[
  \ell_p(\mu) \,=\, \limsup_{t \to 0} t^{1-\frac1p}
  \|\cF\omega_\lin(t)\|_{L^p(\Omega)}~, \qquad 1 \le p \le
  \infty~.
\]
As in the two-dimensional case \cite[Section~2.3.4]{Ga}, a direct
calculation, which exploits some cancellations in the nonlinear term
$u_\lin(t) \cdot \nabla\omega_\lin(t)$ for small times, reveals that
$\ell_p(\mu) = 0$ for any $p \in [1,\infty]$. This in turn implies
that $\delta = 0$, since $2C_3R < 1$. Finally, using again the
relation $\omega_\theta - \omega_\lin = (\cF\omega_\lin -\cF\omega_\theta)
- \cF\omega_\lin$ we conclude that
\begin{equation}\label{zerotime2}
  \limsup_{t \to 0}\|\omega_\theta(t) - \omega_\lin(t)\|_{L^1(\Omega)}
  \,\le\, CR\delta + \ell_1(\mu) \,=\, 0~.
\end{equation}
It follows in particular from \eqref{zerotime2} that $\omega_\theta(t)
\weakto \mu$ as $t \to 0$, because we can use the explicit formula
\eqref{Sdef2} to verify that $\omega_\lin(t) \weakto \mu$ as $t \to 0$.
By construction $\omega_\theta$ is a solution of \eqref{omint}, hence 
of \eqref{omint2}, and both inequalities in \eqref{main2prop} hold. 

Finally, if $\omega_\theta \in C^0((0,T],L^1(\Omega) \cap L^\infty(\Omega))$
is a mild solution on $(0,T)$ satisfying \eqref{main2prop} and such that 
$\omega_\theta(t) \weakto \omega_0$ as $t \to 0$, we can take the
limit $t_0 \to 0$ in \eqref{omint2} and conclude that $\omega_\theta$ 
satisfies \eqref{omint}, so that $\omega_\theta$ coincides with 
the solution constructed by the fixed point argument. 
\end{proof}

\begin{rem}\label{uniqrem}
In Proposition~\ref{propex2} we only claim uniqueness of solutions
under assumption \eqref{main2prop}, which means (after restricting the
existence time) that $\omega_\theta$ belongs to the ball $B_R \subset
X_T$ where the fixed point argument works. As in the two-dimensional
case, one may conjecture that uniqueness holds among all
solutions $\omega_\theta \in C^0((0,T),L^1(\Omega) \cap
L^\infty(\Omega))$ such that $\|\omega_\theta(t)\|_{L^1(\Omega)}$ is
uniformly bounded and $\omega_\theta(t) \weakto \mu$ as $t \to 0$. We
hope to come back to that interesting question in a future work.
\end{rem}

\begin{rem}\label{smooth}
The solutions constructed in Propositions~\ref{propex1} and 
\ref{propex2} are in fact smooth for positive times, and satisfy 
the axisymmetric vorticity equation \eqref{omeq2} in the classical 
sense. This can be proved using standard smoothing properties
of the Navier-Stokes equations that are not specific to the 
axisymmetric case, see e.g. \cite{KNSS} and Proposition~\ref{Lpderprop} 
below. 
\end{rem}

\section{A priori estimates and global existence}\label{sec5}

We continue the proof of Theorems~\ref{main1} and \ref{main2} by
showing that the local solutions constructed in Sections~\ref{sec41}
and \ref{sec42} can be extended to global solutions for positive
times. Since we are not interested in the behavior for small times, we
can assume without loss of generality that the initial vorticity is
integrable. Let thus $\omega_\theta \in C^0([0,T],L^1(\Omega)) \cap
C^0((0,T],L^\infty(\Omega))$ be a solution of the integral equation
\eqref{omint}, hence also of the differential equation \eqref{omeq2},
with initial data $\omega_0 \in L^1(\Omega)$. Our goal here is to
derive a priori estimates on various norms of $\omega_\theta$.

\begin{lem}\label{L1lem}
The solution of \eqref{omint} satisfies $\|\omega_\theta(t)\|_{L^1(\Omega)} 
\le \|\omega_0\|_{L^1(\Omega)}$ for all $t \in [0,T]$. Moreover, if 
$\omega_0 \not\equiv 0$, the map $t \mapsto \|\omega_\theta(t)
\|_{L^1(\Omega)}$ is strictly decreasing.
\end{lem}

\begin{proof}
We first assume that $\omega_0 \ge 0$ and $\omega_0 \not\equiv 0$.
By the strong maximum principle, the solution $\omega_\theta(t)$
of \eqref{omeq2} is strictly positive for $t \in (0,T]$. Integrating
by parts and using the fact that $\omega_\theta(t)$ satisfies the
homogeneous Dirichlet boundary condition on $\partial\Omega$,
we easily find
\begin{equation}\label{L1decay}
  \frac{\D}{\D t}\int_\Omega \omega_\theta(r,z,t)\dd r \dd z
  \,=\, -2\int_\R \partial_r \omega_\theta(0,z,t)\dd z \,<\, 0~,
\end{equation}
where the last inequality follows from Hopf's lemma. This proves the
claim for positive solutions. Note that, in deriving \eqref{L1decay},
we did not use the precise expression of the velocity field $u$ in
\eqref{omeq2}. In the general case where $\omega_0$ can change sign,
we decompose $\omega_\theta(t) = \omega_\theta^+(t) - \omega_\theta^-(t)$,
where $\omega_\theta^\pm(t)$ are defined as the solutions of the linear
equations
\begin{equation}\label{ompm}
  \partial_t \omega_\theta^\pm + \div_*(u \,\omega_\theta^\pm)
  \,=\, \Bigl(\partial_r^2 + \partial_z^2 + \frac{1}{r}\partial_r
  - \frac{1}{r^2}\Bigr)\omega_\theta^\pm~,
\end{equation}
with initial data $\omega_\theta^\pm(0) = \max(\pm\omega_0,0) \ge 0$.
Both equations in \eqref{ompm} involve the same velocity field $u$,
which is associated to the full solution $\omega_\theta$ via the
axisymmetric Biot-Savart law \eqref{BSu}. The analogue of
\eqref{L1decay} holds for the solutions $\omega_\theta^\pm(t)$ of
\eqref{ompm}, hence
\begin{align*}
  \|\omega_\theta(t)\|_{L^1(\Omega)} \,&\le\, \int_\Omega
  \Bigl(\omega_\theta^+(r,z,t) + \omega_\theta^-(r,z,t)\Bigr)
  \dd r\dd z \\ \,&\le\, \int_\Omega \Bigl(\omega_\theta^+(r,z,0)
  + \omega_\theta^-(r,z,0)\Bigr)\dd r\dd z  \,=\,
  \|\omega_0\|_{L^1(\Omega)}~, \qquad 0 \le t \le T~.
\end{align*}
For $t > 0$, the second inequality is strict if either
$\omega_\theta^+(0)$ or $\omega_\theta^-(0)$ is nonzero, and if both
quantities are nonzero the first inequality is also strict (by the
strong maximum principle). If $\omega_0 \not\equiv 0$, this proves
that the $L^1$ norm of the solution $\omega_\theta(t)$ is strictly
decreasing at initial time, and a similar argument shows that it is
strictly decreasing over the whole interval $[0,T]$.
\end{proof}

Higher $L^p$ norms of the vorticity $\omega_\theta$ are more difficult
to control, because the velocity field in \eqref{omeq2} does not
satisfy $\div_* u = 0$. As in \cite{La,UY}, we thus consider the
related quantity $\eta = \omega_\theta/r$, which satisfies
Eq.~\eqref{etaeq} with initial data $\eta_0 = \omega_0/r \in
L^1(\R^3)$.  Using the existence result in Proposition~\ref{propex1}
and the weighted estimates on the linear semigroup given in
Proposition~\ref{Sprop2}, it is easy to verify that $\eta \in
C^0([0,T],L^1(\R^3)) \cap C^0((0,T],L^\infty(\R^3))$. Moreover, by
Lemma~\ref{L1lem}, the map $t \mapsto \|\eta(t)\|_{L^1(\R^3)} \equiv
\|\omega_\theta(t)\|_{L^1(\Omega)}$ is decreasing for nonzero
solutions.  Since the advection field $u - (2/r)e_r$ in \eqref{etaeq}
satisfies
\[
  \div\Bigl(u - \frac{2e_r}{r}\Bigr) \,=\, -2\div \frac{e_r}{r}
  \,=\, -4\pi \delta_{r=0} \,\le\, 0~,
\]
a classical method due to Nash \cite{Na} gives the following a priori
estimate\:

\begin{lem}\label{Lplem} {\bf \cite[Lemma~3.8]{FS}}
For any initial data $\eta_0 \in L^1(\R^3)$, the solution
of \eqref{etaeq} satisfies, for $1 \le p \le \infty$,
\begin{equation}\label{Lp1}
  \|\eta(t)\|_{L^p(\R^3)} \,\le\, \frac{C}{t^{\frac32(1-\frac1p)}}
  \,\|\eta_0\|_{L^1(\R^3)}~, \qquad 0 < t \le T~.
\end{equation}
\end{lem}

\noindent
Equivalently, the axisymmetric vorticity $\omega_\theta$
satisfies, for $p \in [1,\infty]$, the weighted estimate
\begin{equation}\label{Lp2}
  \|r^{\frac1p-1} \omega_\theta(t)\|_{L^p(\Omega)} \,\le\,
  \frac{C}{t^{\frac32(1-\frac1p)}}\,\|\omega_0\|_{L^1(\Omega)}~,
  \qquad 0 < t \le T~.
\end{equation}
Using \eqref{Lp2}, we now establish our main a priori estimate on
the solutions of \eqref{omeq}.

\begin{prop}\label{Lpprop}
Any solution $\omega_\theta \in C^0([0,T],L^1(\Omega)) \cap C^0((0,T],
L^\infty(\Omega))$ of \eqref{omint} with initial data $\omega_0 \in 
L^1(\Omega)$ satisfies, for all $p \in [1,\infty]$,  
\begin{equation}\label{Lpest}
  \|\omega_\theta(t)\|_{L^p(\Omega)} \,\le\,
  \frac{C_p(\|\omega_0\|_{L^1(\Omega)})}{t^{1-\frac1p}}~,
  \qquad 0 < t \le T~,
\end{equation}
where $C_p(s) = \cO(s)$ as $s \to 0$.
\end{prop}

\begin{proof}
We can assume without loss of generality that $M := \|\omega_0\|_{L^1(\Omega)}
> 0$. We know from Lemma~\ref{L1lem} that $\|\omega_\theta(t)\|_{L^1(\Omega)}
\le M$ for $t \in [0,T]$, hence \eqref{Lpest} holds for $p = 1$. To prove
\eqref{Lpest} for $p = 2$, we compute
\begin{equation}\label{Lpder}
  \frac{\D}{\D t}\int_\Omega \omega_\theta^2\dd r \dd z \,=\,
  -2\int_\Omega |\nabla \omega_\theta|^2 \dd r \dd z + \int_\Omega
  \Bigl(\frac{u_r}{r} - \frac{1}{r^2}\Bigr)\omega_\theta^2 \dd r \dd z~.
\end{equation}
The celebrated Nash inequality \cite{Na} asserts that
\[
  \int_\Omega \omega_\theta^2 \dd r \dd z \,\le\, C
  \left(\int_\Omega |\omega_\theta| \dd r \dd z\right)
  \left(\int_\Omega |\nabla\omega_\theta|^2 \dd r \dd z\right)^{1/2}
  \,\le\, C M \left(\int_\Omega |\nabla\omega_\theta|^2 \dd r \dd z
  \right)^{1/2}~.
\]
On the other hand, using estimate \eqref{Lp2} with $p = \infty$
and Proposition~\ref{BSprop3}, we obtain
\[
  \Bigl\|\frac{u_r(t)}{r}\Bigr\|_{L^\infty(\Omega)} \,\le\, 
  C \,\|\omega_\theta(t)\|_{L^1(\Omega)}^{1/3} \,\Bigl\|\frac{\omega_\theta(t)}{r}
  \Bigr\|_{L^\infty(\Omega)}^{2/3} \,\le\, \frac{C M}{t}~.
\]
Thus, if we define
\[
  f(t) \,=\, \int_\Omega \omega_\theta(r,z,t)^2\dd r \dd z~,
  \qquad 0 \le t \le T~,
\]
we deduce from \eqref{Lpder} that $f : [0,T] \to \R$ satisfies the
differential inequality
\begin{equation}\label{diffineq}
  f'(t) \,\le\, -\frac{K_1}{M^2}\,f(t)^2 + \frac{K_2 M}{t}\,f(t)~,
  \qquad 0 < t \le T~,
\end{equation}
where $K_1, K_2$ are positive constants. If we set $f(t) = t^\alpha
g(t)$ with $\alpha = K_2 M$, we see that \eqref{diffineq} reduces to
the simpler differential inequality $g'(t) \le -K_1M^{-2}t^\alpha g(t)^2$,
which can be integrated of the time interval $[t_0,t] \subset
(0,T]$ to give the bound
\[
  \frac{1}{g(t)} \,\ge\,  \frac{1}{g(t_0)} + \frac{K_1}{M^2}\,
  \frac{1}{\alpha+1}\Bigl(t^{\alpha+1} - t_0^{\alpha+1}\Bigr)
  \,\ge\, \frac{K_1}{M^2}\,\frac{1}{\alpha+1}\Bigl(t^{\alpha+1}
  - t_0^{\alpha+1}\Bigr) \,\xrightarrow[\,t_0 \to 0\,]{}\,
  \frac{K_1}{M^2}\,\frac{t^{\alpha+1}}{\alpha+1}~.
\]
We conclude that
\[
  \|\omega_\theta(t)\|_{L^2(\Omega)}^2 \,=\, f(t) \,=\, t^\alpha g(t)
  \,\le\, \frac{\alpha+1}{K_1}\,\frac{M^2}{t} \,=\,
  \frac{K_2M+1}{K_1}\,\frac{M^2}{t}~, \qquad 0 < t \le T~,
\]
which proves \eqref{Lpest} for $p = 2$ (hence for $1 \le p \le 2$
by interpolation). To reach the same conclusion for higher values
of $p$, we proceed exactly as in the proof of Proposition~\ref{propex1}.
Using the notations \eqref{MNdef}, we know from Proposition~\ref{Sprop}
that $M_p(T) \le CM$ for any $p \in [1,\infty]$, and from the
argument above that $N_p(T) \le C(M)$ for $p \in [1,2]$. The
relation \eqref{MNrec} then shows that $N_p(T) \le C(M)$ for all
$p > 2$, and a second iteration gives the desired result for
$p = \infty$ too.
\end{proof}

\begin{rem}\label{globrem}
Proposition~\ref{Lpprop} shows in particular that the $L^p$ norms of
the vorticity $\omega_\theta(t)$ cannot blow up in finite time. In
view of Remark~\ref{Texist}, this implies that all solutions
constructed in Sections~\ref{sec41} and \ref{sec42} are global for
positive times, and that the conclusions of Lemma~\ref{L1lem}
and Proposition~\ref{Lpprop} hold for all $t > 0$.
\end{rem}

With Proposition~\ref{Lpprop} at hand, it is straightforward to show
that the solutions of the vorticity equation \eqref{omeq} are smooth
for positive times, and that estimates similar to \eqref{Lpest} hold
for the derivatives too. For later use, we state the following result.

\begin{prop}\label{Lpderprop}
Under the assumptions of Proposition~\ref{Lpprop}, we have
for all $p \in [1,\infty]$\:
\begin{equation}\label{Lpderest}
  \|\nabla\omega_\theta(t)\|_{L^p(\Omega)} \,\le\, 
  \frac{C_p(\|\omega_0\|_{L^1(\Omega)})}{t^{\frac32-\frac1p}}~, \qquad t > 0~,
\end{equation}
where $C_p(s) = \cO(s)$ as $s \to 0$.
\end{prop}

\begin{proof}
It is possible to prove estimate \eqref{Lpderest} by working
directly on the integral representation \eqref{omint}, but we find
it easier to deduce it from Proposition~\ref{Lpprop} using general
smoothing properties of the Navier-Stokes equations. We know from
\eqref{BSest2} and \eqref{Lpest} that the velocity field associated
with the solution $\omega_\theta$ of \eqref{omeq2} satisfies, for
any $t_0 > 0$,
\begin{equation}\label{Lpder1}
  \|u(t_0)\|_{L^\infty(\Omega)} \,\le\, C \|\omega_\theta(t_0)\|_{L^1
  (\Omega)}^{1/2}\,\|\omega_\theta(t_0)\|_{L^\infty(\Omega)}^{1/2} 
  \,\le\, \frac{C(M)}{\sqrt{t_0}}~,
\end{equation}
where $M = \|\omega_0\|_{L^1(\Omega)}$ and $C(s) = \cO(s)$ as $s \to 0$.
On the other hand, there exist positive constants $a$ and $A$ 
such that the solution $u(t)$ of the Navier-Stokes equations 
\eqref{NS3D} in $\R^3$ with data $u(t_0) \in L^\infty(\R^3)$ at 
time $t_0$ satisfies
\begin{equation}\label{Lpder2}
  (t-t_0)\|\nabla^2 u(t)\|_{L^\infty(\R^3)} \,\le\, A 
  \|u(t_0)\|_{L^\infty(\R^3)}~, \qquad t_0 < t < t_0 + a 
  \|u(t_0)\|_{L^\infty(\R^3)}^{-2}~,
\end{equation}
see \cite[Proposition~4.1]{KNSS}. Here $\nabla^2 u$ denotes the 
collection of all second-order derivatives of $u$. Since
$\|u(t_0)\|_{L^\infty(\Omega)} \equiv \|u(t_0)\|_{L^\infty(\R^3)}$, we
can combine estimates \eqref{Lpder1}, \eqref{Lpder2} by fixing 
$t > 0$ and choosing, for instance, 
\[
  t_0 \,=\, \frac{t}{2}\,\frac{2C(M)^2 + a}{C(M)^2+a}~,
\]
so that $t_0 < t < t_0 + a t_0 C(M)^{-2} \le t_0 + a \|u(t_0)\|_{L^\infty
(\R^3)}^{-2}$. We thus obtain
\begin{equation}\label{Lpder3}
  \|\nabla^2 u(t)\|_{L^\infty(\R^3)} \,\le\, \frac{A}{t-t_0}\,
  \frac{C(M)}{\sqrt{t_0}} \,=\, \frac{\tilde C(M)}{t^{3/2}}~,
\end{equation}
where $\tilde C(s) = \cO(s)$ as $s \to 0$. Using the pointwise 
estimate
\[
  |\nabla \omega_\theta| \,\le\, |\partial_r \omega_\theta| + 
  |\partial_z \omega_\theta| \,\le\, 
  C\Bigl(|\nabla^2 u| + \frac{1}{r}\,|\partial_r u_z|\Bigr)~,
\]
and the fact that $\partial_r u_z$ vanishes at $r = 0$, we see that 
\eqref{Lpder3} implies \eqref{Lpderest} with $p = \infty$.
The case $p < \infty$ easily follows by interpolation, in 
view of \eqref{Lpest}.  
\end{proof}

\section{Long-time behavior}\label{sec6}

In this final section, we study the long-time behavior of the
solutions of the axisymmetric vorticity equation \eqref{omeq}
constructed in Sections~\ref{sec4} and \ref{sec5}. In particular we
prove estimate \eqref{asym2}, and we obtain the asymptotic formula
\eqref{Mnon-V} in the particular case where the initial vorticity has a
definite sign and a finite impulse in the sense of \eqref{cM-def}. 
This will conclude the proof of Theorems~\ref{main1} and 
\ref{main2}.

\subsection{Convergence to zero in scale invariant
norms}\label{sec52}

Let $\omega_\theta \in C^0([0,\infty),L^1(\Omega)) \cap
C^0((0,\infty),L^\infty(\Omega))$ be a global solution of the
vorticity equation \eqref{omeq2}, with initial data $\omega_0 \in
L^1(\Omega)$. We know from Lemma~\ref{L1lem} that the $L^1$ norm
$\|\omega_\theta(t)\|_{L^1(\Omega)}$ is a decreasing function of
time, and our goal here is to prove that this quantity actually
converges to zero as $t \to \infty$. We first show that
$\omega_\theta(r,z,t)$ is essentially confined, for large times,
in a ball of radius $\cO(\sqrt{t})$ in $\Omega$. 

\begin{prop}\label{propconfi}
Let $\omega_0 \in L^1(\Omega)$ and $M = \|\omega_0\|_{L^1}$.
For any $\epsilon > 0$, there exist positive constants
$K_3(\epsilon,\omega_0)$ and $K_4(\epsilon,M)$ such that
the solution of \eqref{omeq2} with initial data $\omega_0$
satisfies, for all $t \ge 0$,
\begin{equation}\label{confinest}
  \int_{\Omega(t)} |\omega_\theta(r,z,t)| \dd r \dd z \,\le\, \epsilon~,
  \quad \hbox{where}\quad \Omega(t) \,=\, \Bigl\{(r,z) \in \Omega
  \,\Big|\,\sqrt{r^2+z^2} \ge K_3 + K_4\sqrt{t}\Bigr\}~.
\end{equation}
\end{prop}

\begin{proof}
By the maximum principle, it is sufficient to establish
\eqref{confinest} for positive solutions of \eqref{omeq2} 
(in the general case, the result follows by decomposing 
$\omega_\theta$ as in the proof of Lemma~\ref{L1lem}). We thus 
assume that $\omega_0 \ge 0$ and that $M = \int_\Omega \omega_0
\dd r\dd z > 0$. The only property of the velocity field 
that will be used to obtain \eqref{confinest} is the a priori 
estimate \eqref{Lpder1}. 

We first prove confinement in the radial direction. For $R \ge 0$ and
$t \ge 0$, we define
\[
  f(R,t) \,=\, \int_R^\infty \left\{\int_\R \omega_\theta(r,z,t)
  \dd z \right\}\dd r~.
\]
Then $f(R,t)$ is a non-increasing function of $R$ such that $f(0,t) =
\|\omega_\theta(t)\|_{L^1(\Omega)}$ and $f(R,t) \to 0$ as $R \to \infty$.
Moreover, using \eqref{omeq2}, it is easy to verify that $f$ satisfies
the evolution equation
\begin{equation}\label{evolf}
  \partial_t f(R,t) \,=\, \partial_R^2 f(R,t) + \frac1R\,
  \partial_R f(R,t) + \int_\R u_r(R,z,t) \omega_\theta(R,z,t)\dd z~,
  \qquad R > 0~.
\end{equation}
In view of \eqref{Lpder1}, we have the estimate
\begin{equation}\label{evolf1}
  \int_\R u_r(R,z,t) \omega_\theta(R,z,t)\dd z \,\le\,
  \frac{C(M)}{\sqrt{t}}\int_\R \omega_\theta(R,z,t)\dd z
  \,=\, -\frac{C(M)}{\sqrt{t}}\,\partial_Rf(R,t)~,
\end{equation}
for some positive constant $C(M)$. Since $\partial_R f \le 0$, we 
deduce from \eqref{evolf}, \eqref{evolf1} that
\begin{equation}\label{evolf2}
 \partial_t f(R,t) \,\le\, \partial_R^2 f(R,t) -\frac{C(M)}{\sqrt{t}}
 \,\partial_Rf(R,t)~, \qquad R > 0~.
\end{equation}
Solving the differential inequality \eqref{evolf2}, with homogeneous
Neumann boundary condition at $R = 0$, is a straightforward task. For
instance, if we extend $f(\cdot,t)$ to the whole real line by setting
$\overline{f}(R,t) = f(0,t)$ for $R \le 0$, the extension satisfies
inequality \eqref{evolf2} for all $R \in \R$. We deduce that
\[
  f(R,t) \,\le\, g(R-2C(M)\sqrt{t},t)~, \qquad R \ge 0~, \quad
  t \ge 0~,
\]
where $g$ is the solution of the heat equation $\partial_t g=
\partial_R^2 g$ on $\R$ with initial data $g(R,0) = \overline{f}(R,0)$.
Given any $\epsilon > 0$, we choose $R_0 > 0$ large enough so that
$f(R_0,0) \le \epsilon$. If $R > R_0$ we estimate
\begin{align*}
  g(R,t) \,&=\, \frac{1}{\sqrt{4\pi t}}\int_{-\infty}^{R_0}
  e^{-\frac{(R-r)^2}{4t}} \,\overline{f}(r,0)\dd r +
  \frac{1}{\sqrt{4\pi t}}\int_{R_0}^\infty  e^{-\frac{(R-r)^2}{4t}}f(r,0)\dd r \\
  \,&\le\,  \frac{e^{-\frac{(R-R_0)^2}{4t}}}{\sqrt{4\pi t}} \int_{-\infty}^{R_0}
  e^{-\frac{(R_0-r)^2}{4t}} M\dd r + \frac{\epsilon}{\sqrt{4\pi t}}
  \int_{R_0}^\infty  e^{-\frac{(R-r)^2}{4t}}\dd r \\
  \,&\le\, M \,e^{-\frac{(R-R_0)^2}{4t}} + \epsilon~.
\end{align*}
The right-hand side is smaller than $2\epsilon$ if $R \ge R_0 +
2\sqrt{t}\log(M/\epsilon)^{1/2}$. Summarizing, we have shown that
$f(R,t) \le 2\epsilon$ provided $R \ge R_0 + K\sqrt{t}$, with
$K = 2C(M) + 2 \log(M/\epsilon)^{1/2}$.

The argument is similar for the confinement in the vertical
direction. For $Z \in \R$ and $t \ge 0$, we define
\[
  h_+(Z,t) \,=\, \int_Z^\infty \left\{\int_0^\infty \omega_\theta(r,z,t)
  \dd r \right\}\dd z~, \qquad
  h_-(Z,t) \,=\, \int_{-\infty}^Z \left\{\int_0^\infty \omega_\theta(r,z,t)
  \dd r \right\}\dd z~.
\]
Then $Z \mapsto h_+(Z,t)$ is non-increasing, $Z \mapsto h_-(Z,t)$
is non-decreasing, and we have the differential inequalities
\[
  \partial_t h_\pm(Z,t) \,\le\, \partial_Z^2 h_\pm(Z,t)
  \mp \frac{C(M)}{\sqrt{t}}\,\partial_Z h_\pm(Z,t)~,
\]
which allow us to compare $h_\pm(Z,t)$ with suitably translated
solutions of the one-dimensional heat equation. Proceeding as
above, we find that, for any $\epsilon > 0$, there exist positive
constants $Z_0$ and $K$ such that $h_+(Z,t) \le 2\epsilon$ if
$Z \ge Z_0 + K\sqrt{t}$, and $h_-(Z,t) \le 2\epsilon$ if $Z \le
-Z_0 - K\sqrt{t}$. It follows that
\[
  \int_0^\infty \int_{|z| \ge {Z_0+K\sqrt{t}}} \,\omega_\theta(r,z,t)
  \dd z \dd z \,\le\, 4\epsilon~, \qquad t \ge 0~.
\]
Combining this result with the previous estimate on $f(R,t)$,
we obtain \eqref{confinest}.
\end{proof}

\begin{prop}\label{L1zero}
For any initial data $\omega_0 \in L^1(\Omega)$, the solution
of the axisymmetric vorticity equation \eqref{omeq2} satisfies 
$\|\omega_\theta(t)\|_{L^1(\Omega)} \to 0$ as $t \to \infty$.
\end{prop}

\begin{proof}
We know from Lemma~\ref{L1lem} that $\|\omega_\theta(t)\|_{L^1(\Omega)}$
converges to some limit $\ell$ as $t \to \infty$. To prove
that $\ell = 0$, we use a standard rescaling argument.
Let $(\lambda_n)_{n \in \N}$ be an increasing sequence of
positive real numbers such that $\lambda_n \to \infty$ as
$n \to \infty$. We define for all $n \in \N$\:
\[
  w_n(r,z,t) \,=\, \lambda_n^2\,\omega_\theta(\lambda_n r,\lambda_n z,
  \lambda_n^2(1+t))~, \qquad (r,z) \in \Omega~, \quad t \ge 0~,
\]
Denoting $w_n(t) = w_n(\cdot,\cdot,t)$, we claim that the sequence
$(w_n(0))_{n \in \N}$ is relatively compact in $L^1(\Omega)$. Indeed,
by Proposition~\ref{propconfi}, for all $\epsilon > 0$ there
exists a compact set $\Omega_0 \subset \Omega$ such that
\begin{equation}\label{Riesz1}
  \sup_{n \in \N} \int_{\Omega\setminus\Omega_0} |w_n(r,z,0)|
  \dd r\dd z \,\le\, \epsilon~.
\end{equation}
In addition, Proposition~\ref{Lpderprop} asserts that
\begin{equation}\label{Riesz2}
  \sup_{n \in \N} \|\nabla w_n(0)\|_{L^\infty(\Omega)} \,=\,
  \sup_{n \in \N} \lambda_n^3 \|\nabla \omega_\theta(\lambda_n^2)
  \|_{L^\infty(\Omega)} \,<\, \infty~.
\end{equation}
Combining \eqref{Riesz1}, \eqref{Riesz2} and using the Riesz criterion
\cite[Theorem~XIII.66]{RS}, we obtain the desired compactness. Thus,
after extracting a subsequence, we can assume that $w_n(0)$ converges
in $L^1(\Omega)$ to some limit $\overline{w}$, which satisfies
$\|\overline{w}\|_{L^1(\Omega)} = \ell$ because
$\|w_n(0)\|_{L^1(\Omega)} = \|\omega_\theta(\lambda_n^2)
\|_{L^1(\Omega)} \to \ell$ as $n \to \infty$.

Now, we fix $t > 0$, and repeating the procedure above we extract yet
another subsequence so that $w_n(t)$ converges in $L^1(\Omega)$ to
some limit $\overline{w}(t)$. By construction, for each $n \in \N$,
the function $w_n(r,z,t)$ solves the vorticity equation \eqref{omeq2}
with initial data $w_n(r,z,0)$.  As in the proof of
Proposition~\ref{propex1}, we thus denote $w_n(t) =
\Sigma(t)w_n(0)$. Taking the limit $n \to \infty$ and using the fact
that the solutions of \eqref{omeq2} depend continuously on the initial data
in $L^1(\Omega)$, see Remark~\ref{Texist}.2, we deduce that
$\overline{w}(t) \,=\, \Sigma(t)\overline{w}$. But since
$\|\overline{w}(t)\|_{L^1(\Omega)} = \|\overline{w}\|_{L^1(\Omega)} =
\ell$, we get a contradiction with the strict decay of the $L^1$
norm established in Lemma~\ref{L1lem}, unless $\ell = 0$.
\end{proof}

\begin{cor}\label{Lpzero}
Under the assumptions of Proposition~\ref{L1zero}, we
have
\begin{equation}\label{Lpasym}
  \lim_{t \to \infty} t^{1-\frac1p} \|\omega_\theta(t)\|_{L^p(\Omega)}
  \,=\, 0~, \qquad 1 \le p \le \infty~.
\end{equation}
\end{cor}

\begin{proof}
We first observe that, for any $\omega_0 \in L^1(\Omega)$,
one has
\begin{equation}\label{Lplin}
  \lim_{t \to \infty} t^{1-\frac1p}\|S(t)\omega_0\|_{L^p(\Omega)}
  \,=\, 0~, \qquad 1 \le p \le \infty~,
\end{equation}
where $S(t)$ is the semigroup associated with the linear equation
\eqref{omlin}. Indeed, it is clear from \eqref{Sdef} that
\eqref{Lplin} holds if $\omega_0$ is continuous and compactly
supported in $\Omega$, and the general case follows by a density
argument, using estimate \eqref{Sest1}.

Next, in view of Proposition~\ref{L1zero}, it is sufficient
to prove \eqref{Lpasym} for small solutions. With the notations
of Section~\ref{sec41}, we thus assume that $C_2\|\omega_0\|_{L^1(\Omega)}
\le R$, where $R > 0$ satisfies $2C_3 R < 1$. Then the solution
$\omega_\theta$ of \eqref{omint} can be constructed by a global
fixed argument in the ball $B_R \subset X$, where the space $X$
is defined by the norm $\|\omega_\theta\|_X = \sup_{t > 0}t^{1/4}
\|\omega_\theta(t)\|_{L^{4/3}(\Omega)}$. For any $p \in [1,\infty]$,
we define
\[
  \sigma_p(t) \,=\, \sup_{s > t} s^{1-\frac1p} \|\omega_\theta(s)\|_{L^p(\Omega)}
  ~, \qquad \hbox{and}\quad \sigma_p \,=\, \lim_{t \to \infty}\sigma_p(t)~.
\]
Clearly $\sigma_{4/3}(t) \le \sigma_{4/3}(0) = \|\omega_\theta\|_X
\le R$ for all $t \ge 0$. In addition, using the integral equation 
\eqref{omint} and estimating the nonlinear term as in \eqref{cFest}
with $p = 4/3$, we find
\begin{align*}
  t^{1/4} \|\omega_\theta(t)\|_{L^{4/3}(\Omega)} \,&\le\,
  t^{1/4} \|S(t)\omega_0\|_{L^{4/3}(\Omega)} \\
  &\quad + t^{1/4} \int_0^{\sqrt{t}} \frac{C}{(t-s)^{3/4}}\,
  \frac{\|\omega_\theta\|_X^2}{s^{1/2}}\dd s + t^{1/4} \int_{\sqrt{t}}^t
  \frac{C}{(t-s)^{3/4}}\,\frac{\sigma_{4/3}(\sqrt{t})^2}{s^{1/2}}\dd s \\
  \,&\le\, t^{1/4} \|S(t)\omega_0\|_{L^{4/3}(\Omega)} + C_3\,t^{-1/4}
  \|\omega_\theta\|_X^2 + C_3\,\sigma_{4/3}(\sqrt{t})^2~, \qquad t > 0~.
\end{align*}
Taking the limsup as $t \to \infty$ and using \eqref{Lplin},
we see that $\sigma_{4/3} \,\le\, C_3\,\sigma_{4/3}^2$. Since we also
know that $C_3\,\sigma_{4/3} \le C_3 R < 1/2$, we conclude that
$\sigma_{4/3} = 0$. With this information at hand, we can return
to the integral equation \eqref{omint} and show, by the same
bootstrap argument as in the proof of Proposition~\ref{propex1},
that $\sigma_p = 0$ for all $p \in [1,\infty]$. This concludes the
proof of \eqref{Lpasym}.
\end{proof}

\subsection{Asympotic behavior of positive solutions
with finite impulse}\label{sec53}

We now consider the particular situation where the initial vorticity
$\omega_0 \in L^1(\Omega)$ is non-negative and has a finite impulse. 
We denote
\begin{equation}\label{Mdef}
  M \,=\, \int_{\Omega} \omega_0(r,z)\dd r\dd z~, \qquad
  \cI \,=\, \int_{\Omega} r^2 \omega_0(r,z)\dd r\dd z~.
\end{equation}
As is well-known, the impulse $\cI$ is conserved for solutions
of \eqref{omeq2}.

\begin{lem}\label{Mconser}
For any non-negative solution of \eqref{omeq2} in $L^1(\Omega)$ with
finite impulse, we have
\[
  \int_{\Omega} r^2 \omega_\theta(r,z,t)\dd r\dd z \,=\,
  \int_{\Omega} r^2 \omega_0(r,z)\dd r\dd z~, \qquad t \ge 0~.
\]
\end{lem}

\begin{proof}
Using \eqref{omeq2} we find by a direct calculation
\begin{align*}
 \frac{\D}{\D t}\int_\Omega r^2\omega_\theta\dd r \dd z \,&=\,
  \int_\Omega r^2\Bigl(\partial_r^2 + \partial_z^2 + \frac{1}{r}\partial_r
  - \frac{1}{r^2}\Bigr)\omega_\theta\dd r\dd z - \int_\Omega r^2
  \div_*(u \omega_\theta) \dd r\dd z \\
  \,&=\, 2\int_\Omega ru_r \omega_\theta \dd r\dd z~.
\end{align*}
The last integral actually vanishes. Indeed, using the explicit
formulas \eqref{BSu} and \eqref{Grdef}, we obtain
\[
  \int_\Omega ru_r \omega_\theta \dd r\dd z \,=\, -\frac{1}{\pi}
  \int_\Omega\int_\Omega \frac{z-\bar z}{r^{1/2}\,\bar r^{1/2}}
  \,F'(\xi^2)\omega_\theta(\bar r,\bar z)\omega_\theta(r,z)
  \dd\bar r\dd\bar z\dd r\dd z \,=\, 0~,
\]
because the integrand is odd with respect to the permutation
$(r,z) \leftrightarrow (\bar r,\bar z)$. This gives the
desired result.
\end{proof}

Under the assumption that $\cI < \infty$, one can obtain precise
information on the long-time behavior of the axisymmetric 
vorticity. We begin with the linear case\:

\begin{lem}\label{Mlindecay}
If $\omega_0 \in L^1(\Omega)$ is non-negative and $\cI < \infty$,
one has
\begin{equation}\label{Mlin}
  \lim_{t \to \infty} t^2 (S(t)\omega_0)(r\sqrt{t},z\sqrt{t}) \,=\,
  \frac{\cI}{16\sqrt{\pi}}\,r\,e^{-\frac{r^2+z^2}{4}}~, \qquad
  (r,z) \in \Omega~,
\end{equation}
where convergence holds in $L^p(\Omega)$ for $1 \le p \le \infty$.
\end{lem}

\begin{proof}
In view of the explicit formula \eqref{Sdef}, we have for all
$(r,z) \in \Omega$\:
\[
  t^2(S(t)\omega_0)(r\sqrt{t},z\sqrt{t}) \,=\, \frac{r}{4\pi}
  \int_\Omega \,K\Bigl(\frac{\sqrt{t}}{r\bar r}\Bigr)\,\exp\Bigl(
  -\frac14\Bigl[\Bigl(r{-}\frac{\bar r}{\sqrt{t}}\Bigr)^2 +
  \Bigl(z{-}\frac{\bar z}{\sqrt{t}}\Bigr)^2\Bigr]\Bigr)
  \,\bar r^2\omega_0(\bar r,\bar z)\dd\bar r\dd\bar z~,
\]
where $K(\tau) = \tau^{3/2}H(\tau)$ is uniformly bounded and
converges to $\sqrt{\pi}/4$ as $\tau \to +\infty$. The 
pointwise result \eqref{Mlin} thus follows from Lebesgue's dominated 
convergence theorem. Convergence in $L^p$ norms is easy to 
prove if $\omega_0$ has compact support in $\Omega$, and can 
be established in the general case by an approximation argument.
\end{proof}

It follows in particular from \eqref{Mlin} that $\|S(t)
\omega_0\|_{L^p(\Omega)} = \cO(t^{-2+\frac1p})$ as $t \to \infty$.
Our last result is the extension of Lemma~\ref{Mlindecay}
to the nonlinear case. 

\begin{prop}\label{Mnonlin}
If $\omega_0 \in L^1(\Omega)$ is non-negative and $\cI < \infty$,
the solution of \eqref{omeq2} with initial data $\omega_0$
satisfies
\begin{equation}\label{Mnon}
  \lim_{t \to \infty} t^2 \omega_\theta(r\sqrt{t},z\sqrt{t},t) \,=\,
  \frac{\cI}{16\sqrt{\pi}}\,r\,e^{-\frac{r^2+z^2}{4}}~, \qquad
  (r,z) \in \Omega~,
\end{equation}
where convergence holds in $L^p(\Omega)$ for all $p \in [1,\infty]$.
Thus $\|\omega_\theta(t)\|_{L^p(\Omega)} = \cO(t^{-2+\frac1p})$
as $t \to \infty$.
\end{prop}

\begin{proof}
We estimate the integral term in \eqref{omint} in the following
way. First, using \eqref{Sweight2} with $\alpha = 0$ and
$\beta = 1$, we write
\[
  \Bigl\|\int_0^{t/2} S(t-s)
  \div_*(u(s)\omega_\theta(s))\dd s\Bigr\|_{L^1(\Omega)}
  \,\le\, \int_0^{t/2} \frac{C}{t-s}\,\|u(s)\|_{L^\infty(\Omega)}
  \|r\omega_\theta(s)\|_{L^1(\Omega)}\dd s~.
\]
From \eqref{BSest4} we have $\|u\|_{L^\infty}\|r\omega_\theta\|_{L^1}
\le C \|r\omega_\theta\|_{L^1}^{3/2} \|\omega_\theta/r\|_{L^\infty}^{1/2} \le
C \|\omega_\theta\|_{L^1}^{3/4} \|r^2\omega_\theta\|_{L^1}^{3/4}
\|\omega_\theta/r\|_{L^\infty}^{1/2}$, hence using Lemmas~\ref{L1lem}
and \ref{Lplem} we obtain
\[
  \|u(s)\|_{L^\infty(\Omega)}\|r\omega_\theta(s)\|_{L^1(\Omega)} \,\le\,
  C\,\frac{M^{1/2}\,\cI^{3/4}}{s^{3/4}}\,\|\omega_\theta(s)\|_{L^1}^{3/4}~,
\]
where $M, \cI$ are defined in \eqref{Mdef}. Similarly, we find
\[
  \Bigl\|\int_{t/2}^t S(t-s)
  \div_*(u(s)\omega_\theta(s))\dd s\Bigr\|_{L^1(\Omega)}
  \,\le\, \int_{t/2}^t \frac{C}{(t-s)^{1/2}}\,\|u(s)\|_{L^\infty(\Omega)}
  \|\omega_\theta(s)\|_{L^1(\Omega)}\dd s~,
\]
where
\[
  \|u(s)\|_{L^\infty(\Omega)}\|\omega_\theta(s)\|_{L^1(\Omega)}
  \,\le\, C\,\frac{M^{1/2}\,\cI^{1/4}}{s^{3/4}}\,\|\omega_\theta(s)\|_{L^1}^{5/4}~.
\]
Thus it follows from \eqref{omint} that
\begin{align}\nonumber
  \|\omega_\theta(t)\|_{L^1(\Omega)} \,&\le\, \|S(t)\omega_0\|_{L^1(\Omega)}
  + C M^{1/2}\,\cI^{3/4}\int_0^{t/2}\frac{\|\omega_\theta(s)\|_{L^1}^{3/4}}{
  (t-s)\,s^{3/4}}\dd s \\ \label{Mnon1}
  &\quad + C M^{1/2}\,\cI^{1/4}\int_{t/2}^t\frac{
  \|\omega_\theta(s)\|_{L^1}^{5/4}}{(t-s)^{1/2}\,s^{3/4}}\dd s~.
\end{align}
Since $\|S(t)\omega_0\|_{L^1(\Omega)} = \cO(t^{-1})$ as $t \to \infty$
by Lemma~\ref{Mlindecay}, the integral inequality \eqref{Mnon1}
implies (by a straightforward bootstrap argument) that
$\|\omega_\theta(t)\|_{L^1(\Omega)} = \cO(t^{-1})$ as $t \to \infty$. In a
similar way, one can show that $\|\omega_\theta(t)\|_{L^p(\Omega)} =
\cO(t^{-2+1/p})$ as $t \to \infty$ for all $p \in [1,\infty]$. 

To prove \eqref{Mnon}, we fix $t_0 \ge 1$ and consider the integral
equation \eqref{omint2} for $t \ge t_0$. If we bound the nonlinear
term exactly as in \eqref{Mnon1}, using the additional information
that $\|\omega_\theta(t)\|_{L^1(\Omega)} = \cO(t^{-1})$, we obtain the
estimate
\begin{equation}\label{Mnon2}
  t \|\omega_\theta(t) - S(t-t_0)\omega_\theta(t_0)\|_{L^1(\Omega)}
  \,\le\, \frac{C(M,\cI)}{\sqrt{t_0}}~,
\end{equation}
for some positive constant $C(M,\cI)$. We now rescale the solution
$\omega_\theta(r,z,t)$ as in \eqref{Mnon} and take the limit as $t \to
\infty$. Using \eqref{Mnon2} and Lemma~\ref{Mlindecay}, we
obtain
\begin{equation}\label{Mnon3}
  \limsup_{t \to \infty} \|t^2 \omega_\theta(\cdot\sqrt{t},\cdot\sqrt{t},t)
  - \Phi\|_{L^1(\Omega)} \,\le\, \frac{C(M,\cI)}{\sqrt{t_0}}~,
\end{equation}
where $\Phi : \Omega \to \R$ denotes the function defined by the
expression in the right-hand side of \eqref{Mnon}. If we take the
limit $t_0 \to \infty$ in \eqref{Mnon3}, we see that \eqref{Mnon}
holds if convergence is understood in $L^1(\Omega)$. A similar
argument shows that convergence also holds $L^p(\Omega)$ for all $p \in
[1,\infty]$.
\end{proof}


\begin{thebibliography}{99}
\setlength{\itemsep}{-0.4mm}

\bibitem{Ab} H. Abidi,
R\'esultats de r\'egularit\'e de solutions axisym\'etriques pour
le syst\`eme de Navier-Stokes, Bull. Sci. Math. {\bf 132} (2008),
592--624 (in French).

\bibitem{AHK}
H. Abidi, T. Hmidi, S. Keraani,
On the global well-posedness for the axisymmetric Euler equations,
Math. Ann. {\bf 347} (2010), 15--41.

\bibitem{BA} M. Ben-Artzi,
Global solutions of two-dimensional Navier-Stokes and Euler equations,
Arch. Rational Mech. Anal. {\bf 128} (1994), 329--358.

\bibitem{Br} H. Brezis,
Remarks on the preceding paper by M. Ben-Artzi: "Global
solutions of two-dimensional Navier-Stokes and Euler equations'',
Arch. Rational Mech. Anal. {\bf 128} (1994), 359--360.

\bibitem{FS} H. Feng and V. \v Sver\'ak,
On the Cauchy problem for axi-symmetric vortex rings,
Arch. Rational Mech. Anal. {\bf 215} (2015), 89--123.

\bibitem{GG} I. Gallagher and Th. Gallay,
Uniqueness for the two-dimensional Navier-Stokes equation with a
measure as initial vorticity,
Math. Ann. {\bf 332} (2005), 287--327.

\bibitem{Ga} Th. Gallay,
Stability and interaction of vortices in two-dimensional viscous
flows,  Discr. Cont. Dyn. Systems Ser. S {\bf 5} (2012), 1091-1131.

\bibitem{GW1} Th. Gallay and C. E. Wayne,
Invariant manifolds and the long-time asymptotics of
the Navier-Stokes and vorticity equations on $\R^2$,
Arch. Ration. Mech. Anal. {\bf 163} (2002), 209--258.

\bibitem{GW2} Th. Gallay and C. E. Wayne,
Global Stability of vortex solutions of the two-dimensional
Navier-Stokes equation, Comm. Math. Phys. {\bf 255} (2005), 97--129.

\bibitem{GW3} Th. Gallay and C. E. Wayne,
Long-time asymptotics of the Navier-Stokes and vorticity 
equations on $\R^3$, Phil. Trans. R. Soc. Lond. A {\bf 360} 
(2002), 2155--2188. 

\bibitem{GMO} Y.~Giga, T.~Miyakawa and H.~Osada,
Two-dimensional Navier-Stokes flow with measures as initial
vorticity, Arch. Rational Mech. Anal. {\bf 104} (1988), 223--250.

\bibitem{GM} Y. Giga and T. Miyakawa,
Navier-Stokes flow in $\R^3$ with measures as initial vorticity and
Morrey spaces, Comm. Partial Diff. Equations {\bf 14} (1989), 577--618.

\bibitem{Ka1} T.~Kato,
Strong $L^p$-solutions of the Navier-Stokes equation in $\R^m$,
with applications to weak solutions, Math. Z. {\bf 187} (1984),
471--480.

\bibitem{Ka2} T.~Kato,
The Navier-Stokes equation for an incompressible fluid in $\R^2$
with a measure as the initial vorticity, Differential Integral
Equations {\bf 7} (1994), 949--966.

\bibitem{KNSS} G. Koch, N. Nadirashvili, G. Seregin, and
V. \v Sver\'ak, Liouville theorems for the Navier-Stokes equations 
and applications, Acta Math. {\bf 203} (2009), 83--105. 

\bibitem{KT} H. Koch and D. Tataru,
Well-posedness for the Navier-Stokes equations, Adv. Math.
{\bf 157} (2001), 22--35.

\bibitem{La} O. Ladyzhenskaya,
Unique solvability in the large of the three-dimensional Cauchy problem
for the Navier-Stokes equations in the presence of axial symmetry,
Zap. Nauchn. Semin. Leningr. Otd. Mat. Inst. Steklova {\bf 7} (1968),
155--177 (in Russian).

\bibitem{LMNP}
S. Leonardi, J. M\'alek, J. Ne\v{c}as, M. Pokorn\'y,
On axially symmetric flows in $\R^3$,
Z. Anal. Anwendungen {\bf 18} (1999), 639--649.

\bibitem{LW}
Jian-Guo Liu and Wei-Cheng Wang,
Characterization and regularity for axisymmetric solenoidal vector
fields with application to Navier-Stokes equation,
SIAM J. Math. Anal. {\bf 41} (2009), 1825--1850.

\bibitem{Na} J. Nash, 
Continuity of solutions of parabolic and elliptic equations, 
Amer. J. Math. {\bf 80} (1958), 931--954. 

\bibitem{RS}
M. Reed and B. Simon,
{\em Methods of modern mathematical physics. {IV}. Analysis of operators},
Academic Press, New York, 1978.

\bibitem{Sv} V. \v Sver\'ak, {\it Selected Topics in Fluid
Mechanics}, lectures notes of an introductory graduate course
taught in 2011/2012, available at the following URL\:\\
{\tt http://www.math.umn.edu/$\sim$sverak/course-notes2011.pdf}

\bibitem{UY} M. Ukhovskii and V. Yudovich,
Axially symmetric flows of ideal and viscous fluids filling the
whole space, J. Appl. Math. Mech. {\bf 32} (1968), 52--61.

\end{thebibliography}
\end{document}